\documentclass[a4paper]{article}
\usepackage{amsmath,amsfonts,amssymb,amsthm}
\usepackage{enumerate}
\usepackage{graphicx}
\usepackage{multirow}
\usepackage{comment}
\usepackage[
backend=biber,
style=alphabetic,
sorting=nty
]{biblatex}
\usepackage{dsfont}
\usepackage{color}
\usepackage{subfigure}
\newtheorem{propo}{Proposition}
\newtheorem{lemma}[propo]{Lemma}
\newtheorem{remark}[propo]{Remark}

\newtheorem{theo}[propo]{Theorem}
\newtheorem{coro}[propo]{Corollary}
\newtheorem{ex}[propo]{Example}

\DeclareMathOperator{\sign}{sgn}
\DeclareMathOperator{\spn}{span}
\DeclareMathOperator{\conv}{conv}

\begin{document}
\title{Real eigenstructure of regular simplex tensors}
\author{Adam Czapli\'{n}ski\footnote{Universit\"at Siegen, Department Mathematik, Walter-Flex-Stra{\ss}e 3, 57068 Siegen, Germany. E-Mail: adam.czaplinski@uni-siegen.de},
Thorsten Raasch\footnote{Universit\"at Siegen, Department Mathematik, Walter-Flex-Stra{\ss}e 3, 57068 Siegen, Germany. E-Mail: raasch@mathematik.uni-siegen.de} and Jonathan Steinberg\footnote{Universit\"at Siegen, Department Physik, Walter-Flex-Stra{\ss}e 3, 57068 Siegen, Germany. E-Mail: steinberg@physik.uni-siegen.de}}
\maketitle

\begin{abstract}
We are concerned with the eigenstructure of supersymmetric tensors. Like in the matrix case, normalized tensor eigenvectors are fixed points of the tensor power iteration map. However, unless the given tensor is orthogonally decomposable, some of these fixed points may be repelling and therefore be undetectable by any numerical scheme. In this paper, we consider the case of regular simplex tensors whose symmetric decomposition is induced by an overcomplete, equiangular set of $n+1$ vectors from $\mathds R^n$. We discuss the full real eigenstructure of such tensors, including the robustness analysis of all normalized eigenvectors. As it turns out, regular simplex tensors exhibit robust as well as non-robust eigenvectors which, moreover, only partly coincide with the generators from the symmetric tensor decomposition.
\end{abstract}

\section{Introduction}
In the last decades, the mathematical analysis and the efficient numerical treatment of tensors, i.e., of multivariate data arrays with a typically large number of modes, has received much attention (see \cite{BBK18,SC14} and references therein). The scientific interest in tensor analysis has been driven by manifold applications, ranging from high-dimensional computational chemistry and physics \cite{CO00,GO03}, neuroscience \cite{GMS21}, algebraic statistics \cite{PS05,ST09} and computer vision \cite{GGCM21} to the algorithmic knowledge retrieval from large datasets.

Both theoretically and practically, a multimodal dataset can only be handled efficiently after imposing a certain structural representation which typically also encodes many of its geometric properties, like symmetries or other correlations between the tensor entries. In the case of supersymmetric tensors $\mathcal T\in(\mathds R^n)^d$, a popular storage format is the symmetric decomposition
\begin{equation}\label{eq:symmdecompintro}
\mathcal T=\sum_{k=1}^r \lambda_k\mathbf v_k^{\otimes d},
\end{equation}
being a linear combination of the $d$-fold outer products of certain $n$-dimensional, normalized vectors $\mathbf v_1,\ldots,\mathbf v_r$, with real weights $\lambda_1,\ldots,\lambda_r$.

The eigenvectors of a supersymmetric tensor $\mathcal T$ are of particular interest in many applications \cite{QCC18}. Inspired by the matrix case ($d=2$), it is a natural question whether there is a relationship between the normalized eigenvectors of a symmetric tensor $\mathcal T$ and the generating vectors $\mathbf v_k$ of its symmetric decomposition \eqref{eq:symmdecompintro}. Unfortunately, both sets of vectors usually only coincide if one imposes additional constraints on the vectors $\mathbf v_k$, like orthogonality. The resulting class of orthogonally decomposable (odeco) tensors is structurally rich (cf. \cite{RO16}), including the fact that all eigenvectors $\mathbf v_k$ of an odeco tensor are attractive fixed points of the tensor power iteration map which is used to numerically solve the tensor eigenvalue problem. Unfortunately, the variety of odeco tensors is of very small dimension, limiting its use in the analysis of generic datasets.

Recently, see \cite{ORS15}, significant progress has been made in the analysis of those symmetric tensors $\mathcal T$ whose generating vectors $\mathbf v_k$ constitute an overcomplete set of vectors which is yet sufficiently close to an orthonormal basis, e.g., a tight frame \cite{W18}. This set of frame decomposable (fradeco) tensors is significantly larger than the odeco class. However, the eigenvectors of a fradeco tensor usually deviate from the underlying frame elements, and it is a straightforward question to ask under which circumstances the eigenvectors of a fradeco tensor are numerically recoverable by the tensor power iteration.

Our work is partly motivated by the recent paper \cite{MRU21}. In loc. cit. it was shown that if $\mathbf v_k$ is an eigenvector with non-zero eigenvalue and if the order $d$ of the tensor is sufficiently large, this very eigenvector is robust. Moreover, some sufficient conditions were given which imply that $\mathbf v_k$ is indeed an eigenvector. However, this does not reveal the full eigenstructure of $\mathcal{T}$. As already pointed out in \cite{MRU21}, a fradeco tensor $\mathcal{T}$ might have further eigenvectors different from any of the $\mathbf v_k$. Even worse, some of these eigenvectors may not be robust and therefore not detectable by any numerical scheme.

In this paper, we will focus on the special case of regular simplex tensors whose symmetric decomposition \eqref{eq:symmdecompintro} uses equal weights $\lambda_k=1$ and is induced by an overcomplete, equiangular set of $n+1$ vectors $\mathbf v_k$ from $\mathds R^n$. Based on a reformulation of the tensor eigenvector equation as an algebraic system of equations in the barycentric coordinates of the eigenvector with respect to the frame elements $\mathbf v_k$, we will develop a full analysis of the eigenstructure of a regular simplex tensor with local dimension $n\ge 2$ and an arbitrary number $d\ge 2$ of modes. In order to assess the performance of the tensor power iteration, we will then also study the robustness of all normalized eigenvectors in the cases $n=2$ and $n=3$.

As our analysis shows, apart from some low-dimensional special cases, a regular simplex tensor does have eigenvectors which only partly coincide with the generators of its symmetric tensor decomposition, thereby falsifying Conjecture 4.7 of \cite{MRU21}. If the number of modes $d$ is odd, there are eigenvectors with eigenvalue zero, which correlates with the redundancy of the underlying frame. If $d$ is even, a regular simplex tensor has non-robust eigenvectors. 

Let us sketch the structure of the paper. We begin with Section \ref{prelandnot} in which we collect some basic notions and theorems. 
The next Section \ref{eigenpa} is the technical heart of the paper.
In the generic case $n\ge 2$ and $d\ge 2$, we first translate the eigenpair condition into an equivalent system of algebraic equations in the barycentric coordinates of an eigenvector with respect to the underlying frame. By exploiting the properties of the involved nonlinearities, we can then enumerate all possible normalized eigenpairs. Afterwards, we will discuss the special cases $n=2$ and $n=3$ in more detail.
In Section \ref{robustanalysis} we study the robustness of all normalized eigenvectors with respect to the tensor power fixed point iteration $\varphi$. In the special case $n=2$, we develop sharp estimates for the spectral radius of the Jacobian of $\varphi$ at each normalized eigenvector. Moreover, in the case $n=3$, we provide numerical experiments which cover the robustness of simplex tensor eigenvectors for a variety of mode numbers $d$.

\section{Preliminaries and notation}\label{prelandnot}
\subsection{Matrices and tensors}
A real-valued tensor of order $d\in\mathds N$ and with local dimensions $n_1,\ldots,n_d\in\mathds N$ is a $d$-variate data field $\mathcal T\in\mathds R^{n_1\times\cdots\times n_d}$. If the number of modes $d$ is equal to $1$ or to $2$, $\mathcal T$ is a column vector or a matrix, respectively, and will then be written in boldface notation, e.g. $\mathcal T=\mathbf v\in\mathds R^{n_1}$ or $\mathcal T=\mathbf A\in\mathds R^{n_1\times n_2}$. In the sequel, we will use the special vectors $\mathbf 1_n:=(1,\ldots,1)^\top\in\mathds R^n$, and $\mathbf e_k$ shall denote the $k$-th unit vector in $\mathds R^m$, if $m$ is clear from the context.

We will denote the $(i_1,\ldots,i_d)$-th entry of a tensor $\mathcal T$ by $\mathcal T_{i_1,\ldots,i_d}$, where $1\le i_j\le n_j$, $1\le j\le d$. If all mode dimensions $n_j$ are equal to $n\in\mathds N$, $\mathcal T\in\mathds R^{n\times\cdots\times n}$ is called cubic. Finally, a cubic tensor $\mathcal T\in\mathds R^{n\times\cdots\times n}$ is called (super)symmetric if for all permutations $\sigma$ of $\{1,\ldots,d\}$,
\begin{equation*}
\mathcal T_{i_1,\ldots,i_d}=\mathcal T_{i_{\sigma(1)},\ldots,i_{\sigma(d)}}.
\end{equation*}
The set of all symmetric tensors of order $d$ and local dimension $n$ shall be denoted by $S^d(\mathds R^n)$.

In the following, let the vector space $\mathds R^{n_1\times\cdots\times n_d}$ of all real-valued, $d$-mode tensors be endowed with the Frobenius inner product $\langle\cdot,\cdot\rangle$ and the corresponding norm $\|\cdot\|$, 
\begin{equation*}
\langle\mathcal S,\mathcal T\rangle
:=
\sum_{1\le i_j\le n_j}\mathcal S_{i_1,\ldots,i_d}\mathcal T_{i_1,\ldots,i_d},\quad
\|\mathcal T\|:=\sqrt{\langle\mathcal T,\mathcal T\rangle}.
\end{equation*}

\subsection{Symmetric tensor decomposition}
For a given vector $\mathbf v=(v_1,\ldots,v_n)^\top\in\mathds R^n$, let the $d$-fold outer power of $\mathbf v$ be defined as
\begin{equation*}
\mathbf v^{\otimes d}:=\mathbf v\otimes\cdots\otimes\mathbf v\in\mathds R^{n\times\cdots\times n},\quad
(\mathbf v\otimes\cdots\otimes\mathbf v)_{i_1,\ldots,i_d}:=v_{i_1}\cdots v_{i_d}.
\end{equation*}
Then $\mathbf v^{\otimes d}\in S^d(\mathds R^n)$ is called a symmetric rank-1 tensor. It is well-known that each symmetric tensor $\mathcal T\in S^d(\mathds R^n)$ can be decomposed into a finite sum
\begin{equation}\label{eq:symmdecomp}
\mathcal T=\sum_{k=1}^r \lambda_k\mathbf v_k^{\otimes d}
\end{equation}
where $\mathbf v_1,\ldots,\mathbf v_k\in\mathds R^n$ with $\|\mathbf v_k\|=1$ and $\lambda_k\in\mathds R$, $1\le k\le r$. The smallest number $r\in\mathds N$ for which such a decomposition exists is called the symmetric rank of $\mathcal T$. If $d$ is odd, we may assume that all weights $\lambda_k$ in \eqref{eq:symmdecomp} are positive.

If $\mathcal T\in S^d(\mathds R^n)$ has a decomposition \eqref{eq:symmdecomp} such that $\{\mathbf v_1,\ldots,\mathbf v_r\}\subset\mathds R^n$ is an orthonormal set, then $\mathcal T$ is called \emph{orthogonally decomposable} or, in short, an \emph{odeco tensor}. However, the set of all odeco tensors is relatively small since the symmetric rank of an odeco tensor cannot exceed the local dimension $n$. 

\subsection{Frames and simplex frames}
A family of $r\ge n$ vectors $\{\mathbf v_1,\ldots,\mathbf v_r\}\subset\mathds R^n$ is called a \emph{frame} for $\mathds R^n$ if there exist constants $B\ge A>0$ such that
\begin{equation*}
A\|\mathbf v\|^2\le\sum_{k=1}^r\big|\langle\mathbf v,\mathbf v_k\rangle\big|^2\le B\|\mathbf v\|^2,\quad\text{for all }\mathbf v\in\mathds R^n.
\end{equation*}
A frame $\{\mathbf v_1,\ldots,\mathbf v_r\}$ with equal frame constants $A=B$ is called \emph{tight frame}, and a tight frame $\{\mathbf v_1,\ldots,\mathbf v_r\}$ is called a \emph{unit-norm tight frame} if additionally $\|\mathbf v_k\|=1$ for all $1\le k\le r$. It is well-known that a set $\{\mathbf v_1,\ldots,\mathbf v_r\}$ is a unit-norm tight frame for $\mathds R^n$ if and only if
\begin{equation*}
\mathbf V\mathbf V^\top=A\mathbf I,\quad
\mathbf V:=(\mathbf v_1\cdots\mathbf v_r)\in\mathds R^{n\times r}.
\end{equation*}
Typical examples of unit-norm tight frames in each $\mathds R^n$ are given by orthonormal bases $\{\mathbf v_1,\ldots,\mathbf v_n\}\subset\mathds R^n$, with $r=n$ and $A=1$, and by so-called \emph{simplex frames} $\{\mathbf v_1,\ldots,\mathbf v_{n+1}\}\subset\mathds R^n$, where $r=n+1$, $A=\frac{n+1}{n}$, and $\mathbf v_k\in\mathds R^n$ is given by the orthogonal projection of the $k$-th unit vector $\mathbf e_k\in\mathds R^{n+1}$ onto the orthogonal complement of $\mathds 1_{n+1}=(1,\ldots,1)^\top\in\mathds R^{n+1}$, and subsequent renormalization,
\begin{equation}\label{eq:simplexframeconstruction}
\mathbf v_k
=
\begin{cases}
\sqrt{1+\frac{1}{n}}\mathbf e_k-\frac{1}{n^{3/2}}\big(\sqrt{n+1}-1\big)\mathds 1_n,&1\le k\le n,\\
-\frac{1}{\sqrt{n}}\mathds 1_n,&k=n+1.
\end{cases}
\end{equation}
We have that
\begin{equation}\label{eq:simplexframeoperator}
\mathbf V\mathbf V^\top=\frac{n+1}{n}\mathbf I,
\end{equation}
and the Gramian matrix of all simplex frame elements is given by
\begin{equation}\label{eq:simplexframegramian}
\mathbf V^\top\mathbf V=\big(\langle\mathbf v_k,\mathbf v_j\rangle\big)_{1\le j,k\le n+1}
=
\begin{pmatrix}
1&-\frac{1}{n}&\cdots&-\frac{1}{n}\\
-\frac{1}{n}&1&\ddots&\vdots\\
\vdots&\ddots&\ddots&-\frac{1}{n}\\
-\frac{1}{n}&\cdots&-\frac{1}{n}&1
\end{pmatrix},
\end{equation}
and the nullspace of $\mathbf V^\top\mathbf V$ and $\mathbf V$ is spanned by $\mathds 1_{n+1}$.

\subsection{Tensor eigenvalues}
For a given symmetric tensor $\mathcal T\in S^d(\mathds R^n)$, a vector $\mathbf v\in\mathds R^n\setminus\{\mathbf 0\}$ is called a (real) \emph{eigenvector} of $\mathcal T$ with \emph{eigenvalue} $\mu\in\mathds R$ if
\begin{equation}\label{eq:eigenvector}
\mathcal T\cdot \mathbf v^{d-1}=\mu\mathbf v.
\end{equation}
Here $\mathcal T\cdot\mathbf v^{d-1}$ denotes the partial contraction of $\mathcal T$ by $\mathbf v$ along all but one of the $d$ modes,
\begin{equation*}
(\mathcal T\cdot\mathbf v^{d-1})_j
:=
\sum_{i_1,\ldots,i_{d-1}=1}^n \mathcal T_{i_1,\ldots,i_{d-1},j}v_{i_1}\cdots v_{i_{d-1}}.
\end{equation*}
Due to the symmetry of $\mathcal T$, it is irrelevant which particular modes are used in the $(d-1)$-fold sum. A tuple $(\mathbf v,\mu)$ consisting of an eigenvector $\mathbf v\in\mathds R^n\setminus\{0\}$ and an associated eigenvalue $\mu\in\mathds R$ is called an \emph{eigenpair} of $\mathcal T$. If $\|\mathbf v\|=1$, an eigenpair $(\mathbf v,\mu)$ is called a \emph{normalized eigenpair}. Normalized eigenpairs can be understood as critical points of the tensor energy functional
\begin{equation*}
J(\mathbf v):=\langle\mathcal T,\mathbf v^{\otimes d}\rangle,\quad\mathbf v\in\mathds R^n
\end{equation*}
under the unit norm constraint $\|\mathbf v\|=1$.

The left-hand side of \eqref{eq:eigenvector} being $(d-1)$-homogeneous, we can use that each eigenpair $(\mathbf v,\mu)$ of $\mathcal T$ induces the eigenpairs $(t\mathbf v,t^{d-2}\mu)$, $t\ne 0$. In particular, if $(\mathbf v,\mu)$ is an eigenpair of $\mathcal T$, then
\begin{equation}\label{eq:normalizedeigenpair}
\Big(\frac{\mathbf v}{\|\mathbf v\|},\frac{\mu}{\|\mathbf v\|^{d-2}}\Big)
\end{equation}
is a normalized eigenpair of $\mathcal T$.

\subsection{Tensor power iteration}
By normalizing both sides of \eqref{eq:eigenvector}, we observe that a necessary condition for a normalized vector $\mathbf v\in\mathds R^n$ to be an eigenvector of $\mathcal T\in S^d(\mathds R^n)$ with eigenvalue $\mu>0$ is that $\mathbf v$ is a fixed point of the mapping
\begin{equation}\label{eq:phi0}
\varphi:\mathds R^n\setminus\{\mathbf 0\}\to\mathds R^n\setminus\{\mathbf 0\},\quad
\varphi(\mathbf v):=\frac{\mathcal T\cdot\mathbf v^{d-1}}{\|\mathcal T\cdot\mathbf v^{d-1}\|}.
\end{equation}
More generally, if $(\mathbf v,\mu)$ is a normalized eigenpair of $\mathcal T$ with eigenvalue $\mu\ne 0$, we have
\begin{equation*}
s\varphi(\mathbf v)=\mathbf v,\quad
s=\sign(\mu)=\frac{\mu}{|\mu|}.
\end{equation*}
The associated canonical fixed point iteration
\begin{equation}\label{eq:tpm}
\mathbf v^{(j+1)}:=\varphi(\mathbf v^{(j)}),\quad j=0,1,\ldots
\end{equation}
is called \emph{tensor power iteration} (TPI) \cite{AN14}. In the matrix case $\mathcal T \in S^{2}(\mathbb{R}^{n})$, it is well known that the iteration (\ref{eq:tpm}) converges for any starting point $\mathbf 0 \neq \mathbf{v}^{(0)} \in \mathbb{R}^{n} $  to an eigenvector corresponding to the largest eigenvalue of $\mathcal T$. For general $d \geq 2$, there exists a distinguished class of eigenpairs with respect to their behaviour under this iteration. Let $ \mathcal T \in S^{d}(\mathbb{R}^{n})$ and $\mathbf{v} \in \mathbb{R}^{n}$ a unit vector. We call $\mathbf{v}$ a robust eigenvector of $\mathcal T$ if there exists an open neighbourhood of $\mathbf{v}$ such that the iterates of (\ref{eq:tpm}) starting with any $\mathbf{w}$ from this neighbourhood converge to $\mathbf{v}$. Clearly, non-robust eigenvectors $\mathbf{v}$ cannot be computed by using (TPI) unless the starting point equals $\mathbf{v}$. The robustness of a fixed point of \eqref{eq:phi0} can be quantified by the spectral radius of the Jacobian of $\varphi$ evaluated in that particular fixed point. For a detailed analysis of the robustness of the map \eqref{eq:phi0}, see Section \ref{robustanalysis} for details.

\newpage

\section{Eigenpairs of regular simplex tensors}\label{eigenpa}
We will now perform an exhaustive search for eigenpairs of the simplex tensor
\begin{equation}\label{eq:simplextensor}
\mathcal T:=\sum_{k=1}^{n+1}\mathbf v_k^{\otimes d},
\end{equation}
where $n,d\ge 2$ and $\{\mathbf v_1,\ldots,\mathbf v_{n+1}\}\subset\mathds R^n$ is given by \eqref{eq:simplexframeconstruction}. Our principal idea will be to develop necessary and sufficient algebraic conditions on the barycentric coordinates of a given normalized vector $\mathbf v$ with respect to the generators $\mathbf v_k$ for $\mathbf v$ to be an eigenvector of $\mathcal T$.

\subsection{Generic theory}
We will start our analysis by looking at the generic case $d,n\ge 2$. First of all, by using the linear independence of each $n$-element subset of the simplex frame $\{\mathbf v_1,\ldots,\mathbf v_{n+1}\}$, we can easily deduce the following system of equations for the coordinates of an eigenvector $\mathbf v\in\mathds R^n\setminus\{0\}$ of $\mathcal T$ with respect to the basis $\{\mathbf v_1,\ldots,\mathbf v_n\}$ of $\mathds R^n$.
\begin{lemma}\label{lemma:eigenvectorcoordinates}
$(\mathbf v,\mu)$ is an eigenpair of $\mathcal T$ if and only if $\mathbf v=\sum_{k=1}^n\alpha_k\mathbf v_k$ for some $\alpha_1,\ldots,\alpha_n\in\mathds R$ with $\sum_{k=1}^n|\alpha_k|>0$ and
\begin{equation}\label{eq:eigenvectorcoord1}
\mu\alpha_k
=
\Big(\alpha_k-\frac{1}{n}\sum_{\substack{j=1\\ j\ne k}}^n\alpha_j\Big)^{d-1}-\Big({-}\frac{1}{n}\sum_{j=1}^n\alpha_j\Big)^{d-1},\quad 1\le k\le n.
\end{equation}
\end{lemma}
\begin{proof}
$\{\mathbf v_1,\ldots,\mathbf v_n\}$ is a basis for $\mathds R^n$, so that each eigenvector $\mathbf v\in\mathds R^n\setminus\{\mathbf 0\}$ has a unique representation $\mathbf v=\sum_{k=1}^n\alpha_k\mathbf v_k$ with $\sum_{k=1}^n|\alpha_k|>0$. By inserting this representation into the eigenvector equation \eqref{eq:eigenvector}, and by using the normalization $\|\mathbf v_k\|=1$, \eqref{eq:simplexframegramian} and $\sum_{k=1}^{n+1}\mathbf v_k=\mathbf 0$, we obtain that
\begin{align*}
\mu\sum_{k=1}^n\alpha_k\mathbf v_k
&=
\sum_{k=1}^{n+1}\Big\langle\mathbf v_k,\sum_{j=1}^n\alpha_j\mathbf v_j\Big\rangle^{d-1}\mathbf v_k\\
&=
\sum_{k=1}^n\Big(\alpha_k-\frac{1}{n}\sum_{1\le j\ne k\le n}\alpha_j\Big)^{d-1}\mathbf v_k
+\Big({-}\frac{1}{n}\sum_{j=1}^n\alpha_j\Big)^{d-1}\mathbf v_{n+1}\\
&=
\sum_{k=1}^n\left(\Big(\alpha_k-\frac{1}{n}\sum_{\substack{j=1\\ j\ne k}}^n\alpha_j\Big)^{d-1}-\Big({-}\frac{1}{n}\sum_{j=1}^n\alpha_j\Big)^{d-1}\right)\mathbf v_k,
\end{align*}
which yields \eqref{eq:eigenvectorcoord1} after using the linear independence of $\{\mathbf v_1,\ldots,\mathbf v_n\}$.
\end{proof}

By using that each elementary tensor $\mathbf v_k^{\otimes d}$ is weighted equally within $\mathcal T$, we can deduce the following cyclic symmetry between all eigenpairs.
\begin{lemma}\label{lemma:cyclic}
Let $(\mathbf v,\mu)$ be an eigenpair of $\mathcal T$ with $\mathbf v=\sum_{k=1}^n\alpha_k\mathbf v_k$ for certain $\alpha_k\in\mathds R$. Then for each permutation $\sigma$ of $\{1,\ldots,n+1\}$, also the vectors
\begin{equation*}
\sum_{k=1}^n\alpha_k\mathbf v_{\sigma(k)}
\end{equation*}
are eigenvectors of $\mathcal T$ with the same eigenvalue $\mu$.
\end{lemma}
\begin{proof}
$(\mathbf v,\mu)$ is an eigenpair with $\mathbf v=\sum_{k=1}^n\alpha_k\mathbf v_k$, so that \eqref{eq:eigenvectorcoord1} holds true by Lemma \ref{lemma:eigenvectorcoordinates}. Let $\sigma$ be a permutation of $\{1,\ldots,n+1\}$. If $p:=\sigma(n+1)=n+1$, we have $\{1,\ldots,n\}=\{\sigma(1),\ldots,\sigma(n)\}$, so that \eqref{eq:eigenvectorcoord1} holds true for all $\alpha_{\sigma^{-1}(k)}$, $1\le k\le n$, i.e, Lemma \ref{lemma:eigenvectorcoordinates} yields that
\begin{equation*}
\sum_{k=1}^n\alpha_k\mathbf v_{\sigma(k)}=\sum_{k=1}^n\alpha_{\sigma^{-1}(k)}\mathbf v_k
\end{equation*}
is an eigenvector of $\mathcal T$ with eigenvalue $\mu$. If $p=\sigma(n+1)\in\{1,\ldots n\}$, we have $q:=\sigma^{-1}(n+1)\in\{1,\ldots,n\}$. By using that $\sum_{k=1}^{n+1}\mathbf v_k=\mathbf 0$ and because of the equivalence
\begin{equation}\label{eq:cyclichelp}
1\le\sigma^{-1}(k)\le n\;\wedge\;\sigma^{-1}(k)\ne q\quad\Leftrightarrow\quad 1\le k\le n\;\wedge\; k\ne p,
\end{equation}
we can write
\begin{align*}
\sum_{k=1}^n\alpha_k\mathbf v_{\sigma(k)}
&=
\alpha_q\mathbf v_{n+1}
+
\sum_{\substack{k=1\\ k\ne q}}^n\alpha_k\mathbf v_{\sigma(k)}
=
{-}\alpha_q\mathbf v_p
+
\sum_{\substack{k=1\\ k\ne p}}^n(\alpha_{\sigma^{-1}(k)}-\alpha_q)\mathbf v_k.
\end{align*}
Therefore, it remains to prove that \eqref{eq:eigenvectorcoord1} holds true for
\begin{equation*}
\beta_k:=\begin{cases}
\alpha_{\sigma^{-1}(k)}-\alpha_q,&1\le k\ne p\le n,\\
-\alpha_q,&k=p,
\end{cases}
\end{equation*}
because then the claim follows by an application of Lemma \ref{lemma:eigenvectorcoordinates}. If $k=p$, we compute that by means of \eqref{eq:cyclichelp} and \eqref{eq:eigenvectorcoord1},
\begin{align*}
\lefteqn{\Big(\beta_p-\frac{1}{n}\sum_{\substack{j=1\\ j\ne p}}^n\beta_j\Big)^{d-1}
-\Big({-}\frac{1}{n}\sum_{j=1}^n\beta_j\Big)^{d-1}}\\
&=
\Big({-}\alpha_q-\frac{1}{n}\sum_{\substack{j=1\\ j\ne p}}^n(\alpha_{\sigma^{-1}(j)}-\alpha_q)\Big)^{d-1}
-\Big(\frac{1}{n}\alpha_q-\frac{1}{n}\sum_{\substack{j=1\\ j\ne p}}^n(\alpha_{\sigma^{-1}(j)}-\alpha_q)\Big)^{d-1}\\
&=
\Big({-}\frac{1}{n}\alpha_q-\frac{1}{n}\sum_{\substack{j=1\\ j\ne p}}^n\alpha_{\sigma^{-1}(j)}\Big)^{d-1}
-\Big(\alpha_q-\frac{1}{n}\sum_{\substack{j=1\\ j\ne p}}\alpha_{\sigma^{-1}(j)}\Big)^{d-1}\\
&=
\Big({-}\frac{1}{n}\sum_{j=1}^n\alpha_j\Big)^{d-1}
-
\Big(\alpha_q-\frac{1}{n}\sum_{\substack{j=1\\ j\ne q}}^n\alpha_j\Big)^{d-1}\\
&=
-\alpha_q\\
&=\beta_p.
\end{align*}
If $1\le k\le n$ and $k\ne p$, we can argue in a similar way, again using \eqref{eq:cyclichelp}:
\begin{align*}
\lefteqn{\Big(\beta_k-\frac{1}{n}\sum_{\substack{j=1\\ j\ne k}}^n\beta_j\Big)^{d-1}-\Big({-}\frac{1}{n}\sum_{j=1}^n\beta_j\Big)^{d-1}}\\
&=
\Big(\alpha_{\sigma^{-1}(k)}-\alpha_q+\frac{1}{n}\alpha_q-\frac{1}{n}\sum_{\substack{j=1\\ j\notin\{k,p\}}}^n(\alpha_{\sigma^{-1}(j)}-\alpha_q)\Big)^{d-1}\\
&\phantom{=}
-\Big(\frac{1}{n}\alpha_q-\frac{1}{n}\sum_{\substack{j=1\\ j\ne p}}^n(\alpha_{\sigma^{-1}(j)}-\alpha_q)\Big)^{d-1}\\
&=
\Big(\alpha_{\sigma^{-1}(k)}-\frac{1}{n}\alpha_q-\frac{1}{n}\sum_{\substack{j=1\\ j\notin\{k,p\}}}^n\alpha_{\sigma^{-1}(j)}\Big)^{d-1}
-\Big(\alpha_q-\frac{1}{n}\sum_{\substack{j=1\\ j\ne p}}^n\alpha_{\sigma^{-1}(j)}\Big)^{d-1}\\
&=
\Big(\alpha_{\sigma^{-1}(k)}-\frac{1}{n}\sum_{\substack{j=1\\ j\ne \sigma^{-1}(k)}}^n\alpha_j\Big)^{d-1}
-
\Big(\alpha_q-\frac{1}{n}\sum_{\substack{j=1\\ j\ne q}}^n\alpha_j\Big)^{d-1}\\
&=
\alpha_{\sigma^{-1}(k)}-\alpha_q\\
&=\beta_k.
\end{align*}
\end{proof}

In view of the fact that $\mathds R^n$ can be decomposed into the conical hulls
\begin{equation*}
\Big\{\sum_{k=1}^n\alpha_k\mathbf v_{\sigma(k)}:\alpha_k\ge 0\Big\},\quad
\sigma:\{1,\ldots,n+1\}\to\{1,\ldots,n+1\}\text{ permutation},
\end{equation*}
Lemma \ref{lemma:cyclic} tells us that it is sufficient to compute all eigenpairs $(\mathbf v,\mu)$ with eigenvectors from the set
\begin{equation*}
\Big\{\sum_{k=1}^n\alpha_k\mathbf v_k:\alpha_k\ge 0,\sum_{j=1}^n|\alpha_j|>0\Big\}
\end{equation*}
of all nontrivial conical combinations from the linearly independent set $\{\mathbf v_1,\ldots,\mathbf v_n\}$.
What is more, by the $(d-1)$-homogeneity of the right-hand side of \eqref{eq:eigenvectorcoord1}, it is sufficient to consider all convex combinations
\begin{equation}\label{eq:candidates}
\mathbf v=\sum_{k=1}^n s_k\mathbf v_k,\quad 0\le s_k\le 1,\quad\sum_{k=1}^n s_k=1
\end{equation}
from $\{\mathbf v_1,\ldots,\mathbf v_n\}$ as eigenvector candidates. By inserting this very family of vectors into Lemma \ref{lemma:eigenvectorcoordinates}, we obtain the following eigenpair conditions.

\begin{lemma}\label{lemma:candidates}
Let $0\le s_k\le 1$ for $1\le k\le n-1$, and $s_n:=1-\sum_{k=1}^{n-1}s_k$. Then $(\sum_{k=1}^n s_k\mathbf v_k,\mu)$ is an eigenpair of $\mathcal T$ if and only if
\begin{equation}\label{eq:eigenvectorcoord2}
\mu s_k=\frac{1}{n^{d-1}}\Big(\big((n+1)s_k-1\big)^{d-1}-(-1)^{d-1}\Big),
\quad
1\le k\le n.
\end{equation}
\end{lemma}
\begin{proof}
The claim follows by inserting the coordinates $\alpha_k=s_k$ into \eqref{eq:eigenvectorcoord1}, and by exploiting that $\sum_{k=1}^n s_k=1$. 
\end{proof}

By means of the function
\begin{equation}\label{eq:g}
g(s):=\big((n+1)s-1\big)^{d-1}-(-1)^{d-1},\quad s\in\mathds R,\quad d\ge 2,
\end{equation}
we can rewrite the eigencondition \eqref{eq:eigenvectorcoord2} in a compact way as
\begin{equation}\label{eq:eigenvectorcoord3}
\mu n^{d-1} s_k=g(s_k),\quad 1\le k\le n.
\end{equation}
The eigenvalue $\mu$ of an eigenpair $(\sum_{k=1}^n s_k\mathbf v_k,\mu)$, $s_k\ge 0$, $\sum_{k=1}^n s_k=1$, can hence be computed by
\begin{equation}\label{eq:mu3}
\mu=
\frac{1}{n^{d-1}}\sum_{k=1}^n g(s_k).
\end{equation}
However, we still have to check which solutions $s_1,\ldots,s_{n-1}$ the remaining system of eigenvalue conditions
\begin{equation}\label{eq:critical}
\left\{
\begin{aligned}
s_k\sum_{j=1}^n g(s_j)&=g(s_k),\quad 1\le k\le n-1,\\
s_n&=1-\sum_{j=1}^{n-1}s_j,\\
s_k&\ge 0,\quad 1\le k\le n,\\
\end{aligned}
\right.
\end{equation}
does actually have.

\begin{ex}
If $n=2$, \eqref{eq:critical} reads as the single equation in $s=s_1$
\begin{equation}\label{eq:criticaln2}
s\big(g(s)+g(1-s)\big)=g(s),\quad 0\le s\le 1.
\end{equation}
If $n=3$, \eqref{eq:critical} reads as the coupled system of equations in $s=s_1$ and $t=s_2$
\begin{equation}\label{eq:criticaln3}
\left\{
\begin{aligned}
s\big(g(s)+g(t)+g(1-s-t)\big)&=g(s),\\
t\big(g(s)+g(t)+g(1-s-t)\big)&=g(t),
\end{aligned}
\quad
s,t\ge 0,\quad s+t\le 1.
\right.
\end{equation}
\end{ex}

In order to deduce the solution set of \eqref{eq:critical}, let us first prove some fundamental properties of the auxiliary function $g$.
\begin{lemma}\label{lemma:g}
Let $n,d\ge 2$, and let $g$ be defined as in \eqref{eq:g}.
\begin{enumerate}[(i)]
\item We have $g(0)=0$, $g(\frac{1}{n+1})=(-1)^d$, $g(\frac{2}{n+1})=1+(-1)^d$, and $g(1)=n^{d-1}+(-1)^d>0$.
\item If $d\ge 3$, we have $g'(\frac{1}{n+1})=0$.
\item If $d$ is even, $g$ is strictly increasing and we have $g(s)>0$ for all $s>0$.
\item If $d$ is odd, $g$ is strictly convex with a unique local minimum at $s=\frac{1}{n+1}$, and we have $g(\frac{2}{n+1})=0$, so that $g(s)<0$ for $0<s<\frac{2}{n+1}$ and $g(s)>0$ for $\frac{2}{n+1}<s\le 1$. 
\item The polynomial $p:s\mapsto\frac{g(s)}{s}$ is well-defined. If $d=2$, $p(s)=n+1$ is constant. If $d$ is odd, $p$ is strictly increasing on $[0,\infty)$. If $d\ge 4$ is even, there exists a point $s^*\in [\frac{1}{n},\frac{2}{n+1})$ such that $p$ is strictly decreasing on $[0,s^*]$ and strictly increasing on $[s^*,\infty)$. We have $s^*=\frac{1}{n}$ if and only if $(n,d)=(2,4)$.
\item If $d$ is odd, the $m$-variate polynomial
\begin{equation*}
\mathds R^m\owns (s_1,\ldots,s_m)\mapsto g\Big(1-\sum_{k=1}^m s_k\Big)+\sum_{k=1}^m g(s_k)
\end{equation*}
is strictly convex, with unique minimum at $\mathbf s^*:=(\frac{1}{m+1},\ldots,\frac{1}{m+1})$ and value $(m+1)g(\frac{1}{m+1})$, $\mathbf s^*$ lying in the interior of the $m$-dimensional unit simplex
\begin{equation}\label{eq:Deltam}
\Delta_m:=\big\{(s_1,\ldots,s_m):s_k\ge 0,\sum_{j=1}^m s_j\le 1\big\}.
\end{equation}
\end{enumerate}
\end{lemma}
\begin{proof}
\begin{enumerate}[(i)]
\item
The claim directly follows from \eqref{eq:g}.
\item
If $d\ge 3$, we have $g'(s)=(n+1)(d-1)((n+1)s-1)^{d-2}$, so $g'(\frac{1}{n+1})=0$.
\item
If $d$ is even, we see that
\begin{equation*}
g(s)
=\big((n+1)s-1)^{d-1}+1,\quad s\in\mathds R,
\end{equation*}
is a composition of the strictly increasing functions $s\mapsto (n+1)s-1$ and $t\mapsto t^{d-1}+1$. Therefore, $s>0$ implies that $g(s)>g(0)=0$.
\item 
If $d$ is odd, $g$ is stricly convex as a sum of the strictly convex function $s\mapsto ((n+1)s-1)^{d-1}$ and a constant. In view of $g'(0)=(n+1)(d-1)(-1)^d<0$ and of $g(1)>0$, see (i), the convexity of $g$ implies that there exists exactly one further zero of $g$ in the open interval $(0,1)$, namely $s=\frac{2}{n+1}$, because the oddity of $d$ and the identity
\begin{equation*}
a^k-b^k=(a-b)\sum_{j=0}^k a^jb^{k-1-j}
\end{equation*}
entail that
\begin{equation*}
g(s)
=
\big((n+1)s-1\big)^{d-1}-1
=
\big((n+1)s-2\big)\sum_{j=0}^{d-2}\big((n+1)s-1\big)^j.
\end{equation*}
By the continuity and strict convexity of $g$, it follows that $g(s)<0$ for $0<s<\frac{2}{n+1}$ and $g(s)>0$ for $\frac{2}{n+1}<s\le 1$.
\item 
In view of $g(0)=0$, see (i), $p(s):=\frac{g(s)}{s}$ defines a polynomial of degree $d-2$, and $p(s)=n+1$ if $d=2$. If $d\ge 3$, we compute that for $s>0$,
\begin{align*}
p'(s)
&=
\frac{g'(s)s-g(s)}{s^2}\\
&=
\frac{(d-1)(n+1)s\big((n+1)s-1\big)^{d-2}+(-1)^{d-1}-\big((n+1)s-1\big)^{d-1}}{s^2}\\
&=
\frac{(d-2)\big((n+1)s-1\big)^{d-1}+(d-1)\big((n+1)s-1\big)^{d-2}+(-1)^{d-1}}{s^2}.
\end{align*}
The derivative of the numerator $g'(s)s-g(s)$ reads as
\begin{align*}
\frac{\mathrm d}{\mathrm ds}\big(g'(s)s-g(s)\big)
&=
g''(s)s
=
(d-1)(d-2)(n+1)^2\big((n+1)s-1\big)^{d-3}s,
\end{align*}
having a single zero at $s=0$ and a $(d-3)$-fold zero at $s=\frac{1}{n+1}$. Hence, if $d=2k+1$ is odd, $k\ge 1$, $g'(s)s-g(s)$ is positive if $s>0$, so $p$ is strictly increasing on $[0,\infty)$. If $d=2k$ is even, $k\ge 2$, $g''(s)s$ is negative on $(0,\frac{1}{n+1}]$ and positive on $(\frac{1}{n+1},\infty)$. 
Therefore, in view of
\begin{equation*}
g'\Big(\frac{1}{n+1}\Big)\frac{1}{n+1}-g\Big(\frac{1}{n+1}\Big)=-1
\end{equation*}
and
\begin{equation*}
g'\Big(\frac{2}{n+1}\Big)\frac{2}{n+1}
-g\Big(\frac{2}{n+1}\Big)=2d-4>0,
\end{equation*}
there exists a $s^*\in(\frac{1}{n+1},\frac{2}{n+1})$ such that $p$ is strictly decreasing on $[0,s^*]$ and strictly increasing on $[s^*,\infty)$. Moreover,
\begin{equation*}
g'\Big(\frac{1}{n}\Big)\frac{1}{n}-g\Big(\frac{1}{n}\Big)
=\frac{(d-1)(n+1)-1}{n^{d-1}}-1
\end{equation*}
is nonpositive, and negative if and only if $(n,d)=(2,4)$. This can be seen as follows: Setting
\begin{equation*}
c_{n,d}:=\frac{(d-1)(n+1)-1}{n^{d-1}},\quad n\ge 2,\quad d\ge 4,
\end{equation*}
we observe that
\begin{equation*}
c_{n,4}=\frac{3n+2}{n^3}=\frac{3+\frac{2}{n}}{n^2}
\le\frac{4}{n^2}
\le 1,
\end{equation*}
with equality if and only if $n=2$, and
\begin{align*}
\frac{c_{n,d+1}}{c_{n,d}}
&=
\frac{d(n+1)-1}{n((d-1)(n+1)-1)}
\le
\frac{dn+d-1}{dn+d-1+n}
<1,\quad n\ge 2,\quad d\ge 4.
\end{align*}
Therefore, if $n\ge 2$ and $d=2k\ge 4$ is even, we have $\frac{1}{n}\le s^*$, with equality if and only if $(n,d)=(2,4)$.
\item The $m$-variate polynomial
\begin{equation*}
f(\mathbf s):=g\Big(1-\sum_{k=1}^m s_k\Big)+\sum_{k=1}^m g(s_k),\quad
\mathbf s=(s_1,\ldots,s_m),
\end{equation*}
is convex as a sum of $m+1$ convex functions. $f$ is strictly convex because if $\mathbf s,\mathbf s'\in\mathds R^m$ with $\mathbf s\ne\mathbf s'$, we have $s_k\ne s_k'$ for at least one $1\le k\le m$, so that for each $0<\lambda<1$, the strict convexity of $g$ implies that
\begin{align*}
\lefteqn{
f\big(\lambda\mathbf s+(1-\lambda)\mathbf s')
}\\
&=
g\Big(1-\sum_{k=1}^m\big(\lambda s_k+(1-\lambda)s_k'\big)\Big)
+\sum_{k=1}^m g\big(\lambda s_k+(1-\lambda)s_k'\big)\\
&=
g\Big(\lambda\Big(1-\sum_{k=1}^m s_k\Big)+(1-\lambda)\Big(1-\sum_{k=1}^m s_k'\Big)\Big)
+
\sum_{k=1}^m g\big(\lambda s_k+(1-\lambda)s_k'\big)\\
&<
\lambda g\Big(1-\sum_{k=1}^m s_k\Big)
+(1-\lambda)g\Big(1-\sum_{k=1}^m s_k'\Big)
+\sum_{k=1}^m\big(\lambda g(s_k)+(1-\lambda)g(s_k')\big)\\
&=
\lambda f(\mathbf s)+(1-\lambda)f(\mathbf s').
\end{align*}
$f$ is bounded from below because $g$ is, and the minimality condition
\begin{equation*}
\mathbf 0=\nabla f(\mathbf s^*)=\Big(g'(s_k^*)-g'\Big(1-\sum_{j=1}^m s_j^*\Big)\Big)_{1\le k\le m}
\end{equation*}
together with the injectivity of $g'$ imply that $s_k^*=\frac{1}{m+1}$ for all $1\le k\le m$, and hence
$f(\mathbf s^*)=(m+1)g(\frac{1}{m+1})$. 
\end{enumerate}
\end{proof}

By means of Lemma \eqref{lemma:g}, we are now able to enumerate all solutions of the system \eqref{eq:critical}, i.e., all zeros $\mathbf s:=(s_1,\ldots,s_{n-1})$ of the vector function
\begin{equation}\label{eq:h}
\mathbf h(\mathbf s):=
\Big(s_k\Big(g\Big(1-\sum_{j=1}^{n-1}s_j\Big)+\sum_{j=1}^{n-1}g(s_j)\Big)-g(s_k)\Big)_{1\le k\le n-1}
\end{equation}
in the $(n-1)$-dimensional unit simplex $\Delta_{n-1}=\conv\{\mathbf 0,\mathbf e_1,\ldots,\mathbf e_{n-1}\}$ from \eqref{eq:Deltam}.

\begin{propo}\label{propo:h}
Let $d\ge 2$, and let $h$ be defined as in \eqref{eq:h}.
\begin{enumerate}[(i)]
\item $\mathbf h$ vanishes at least at those $\mathbf s\in\Delta_{n-1}$ such that with $s_n:=1-\sum_{k=1}^{n-1} s_k$, there exists a nonempty subset $K\subseteq\{1,\ldots n\}$ and
\begin{equation}\label{eq:hzeros}
s_k=\begin{cases}
\frac{1}{|K|},&k\in K,\\
0,&k\in\{1,\ldots,n\}\setminus K.
\end{cases}
\end{equation}
\item If $d=2$, $\mathbf h$ is identically zero.
\item If $d$ is odd, there are no further zeros of $\mathbf h$ in $\Delta_{n-1}$ than those from (i).
\item If $d\ge 4$ is even, $\mathbf h$ vanishes at $\mathbf s\in\Delta_{n-1}$ if and only if with $s^*\in [\frac{1}{n},\frac{2}{n+1})$ from Lemma \ref{lemma:g}(v) and $s_n:=1-\sum_{k=1}^{n-1}s_k$, there exist disjoint subsets $K_1\subset\{1,\ldots,n\}$ and $K_2\subset\{1,\ldots,n\}$, at least one of these being nonempty, such that either
\begin{equation}\label{eq:evendeigencond1}
K_1=\emptyset,\quad s_k=\begin{cases}
\frac{1}{|K_2|}>s^*,&k\in K_2,\\
0,&k\in\{1,\ldots,n\}\setminus K_2,
\end{cases}
\end{equation}
or
\begin{equation}\label{eq:evendeigencond2}
K_2=\emptyset,\quad s_k=\begin{cases}
\frac{1}{|K_1|}\le s^*,&k\in K_1,\\
0,&k\in\{1,\ldots,n\}\setminus K_1,
\end{cases}
\end{equation}
or
\begin{equation}\label{eq:evendeigencond3}
K_1,K_2\ne\emptyset,\quad s_k=\begin{cases}
s_{k_1},&k\in K_1,\\
s_{k_2},&k\in K_2,\\
0,&k\in\{1,\ldots,n\}\setminus(K_1\cup K_2),
\end{cases}
\end{equation}
where $s_{k_1}\in (0,s^*)$ is a zero of the polynomial
\begin{equation}\label{eq:evendr}
r(s):=\frac{g(s)}{s}-\frac{|K_2|g\big(\frac{1-|K_1|s}{|K_2|}\big)}{1-|K_1|s}
\end{equation}
and
\begin{equation}\label{eq:evendeigencond3b}
s_{k_2}=\frac{1-|K_1|s_{k_1}}{|K_2|}
\end{equation}
is contained in $(s^*,1]$.
\end{enumerate}
\end{propo}
\begin{proof}
\begin{enumerate}[(i)]
\item Let $K\subseteq\{1,\ldots n\}$ be nonempty, and let $s_k\in[0,1]$ be given by \eqref{eq:hzeros}. Then we have
\begin{equation*}
s_k\ge 0,\quad 
\sum_{k=1}^n s_k=1,\quad
\sum_{k=1}^{n-1}s_k=1-s_n\le 1,
\end{equation*}
i.e., $\mathbf s:=(s_1,\ldots,s_{n-1})\in\Delta_{n-1}$. By using that $g(0)=0$, see Lemma \ref{lemma:g}(i), we compute that for all $1\le k\le n$, regardless of whether $k\in K$ or $k\notin K$,
\begin{align*}
s_k\sum_{j=1}^n g(s_j)
=
s_k\sum_{j\in K}g(s_j)
=
s_k|K|g\Big(\frac{1}{|K|}\Big)
=
g(s_k),
\end{align*}
so that $\mathbf h(\mathbf s)=\mathbf 0$.
\item
If $d=2$, we have $g(s)=(n+1)s$. For each $\mathbf s=(s_1,\ldots,s_{n-1})\in\Delta_{n-1}$, we obtain that with $s_n:=1-\sum_{j=1}^{n-1}s_k$,
\begin{equation*}
s_k\sum_{j=1}^n g(s_j)-g(s_k)
=
s_k(n+1)\sum_{j=1}^n s_j-(n+1)s_k
=
0,
\end{equation*}
so that $\mathbf h$ vanishes identically. 
\item Suppose that $d$ is odd and that $\mathbf s\in\Delta_{n-1}$ is a zero of $\mathbf h$. Then with $s_n:=1-\sum_{k=1}^{n-1}s_k$, the index set $K:=\big\{1\le k\le n:s_k\ne 0\}$ is nonempty, and $g(0)=0$ implies that
\begin{equation*}
\frac{g(s_k)}{s_k}=\sum_{j=1}^n g(s_j)=\sum_{j\in K}g(s_j),\quad k\in K.
\end{equation*}
The left-hand side of this equation is strictly increasing in $s_k$, and the right-hand side is strictly convex in $(s_j)_{j\in K}$ with lower bound $|K|g(\frac{1}{|K|})$, see Lemma \ref{lemma:g}(v)/(vi). Therefore, we obtain the lower bound
\begin{equation*}
s_k\ge\frac{1}{|K|},\quad k\in K,
\end{equation*}
which, in view of $\sum_{k\in K}s_k=1$, is only achievable if $s_k$ is of the form \eqref{eq:hzeros}.
\item If $d\ge 4$ is even and $\mathbf s\in\Delta_{n-1}$ is a zero of $\mathbf h$, as in (iii) we see that with $s_n:=1-\sum_{k=1}^{n-1}s_k$ and $K:=\{1\le k\le n:s_k\ne 0\}\ne\emptyset$,
\begin{equation*}
\frac{g(s_k)}{s_k}=\sum_{j=1}^n g(s_j)=\sum_{j\in K}g(s_j),\quad k\in K.
\end{equation*}
Lemma \ref{lemma:g}(v) tells us that for some $s^*\in [\frac{1}{n},\frac{2}{n+1})$, the polynomial $p(s):=\frac{g(s)}{s}$ of degree $d-2$ is strictly decreasing on $[0,s^*]$ and strictly increasing on $[s^*,\infty)$. Let us split $K$ into
\begin{equation*}
K=K_1\cup K_2,\quad
K_1:=\{k\in K:s_k\le s^*\},\quad K_2:=\{k\in K:s_k>s^*\}.
\end{equation*}
At least one of the subsets $K_1$ and $K_2$ is nonempty, because $K$ is. We consider the three possible special cases separately.
\begin{itemize}
\item If $K_1$ is empty, we obtain that by the injectivity of $p$ on $(s^*,\infty)$, there exists $k_2\in K_2$ such that $s_k=s_{k_2}$ for all $k\in K_2$. We obtain that
\begin{equation*}
\frac{g(s_{k_2})}{s_{k_2}}=\sum_{j\in K}g(s_j)=|K_2|g(s_{k_2})
\end{equation*}
and hence
\begin{equation*}
s_k=\frac{1}{|K_2|},\quad k\in K_2,
\end{equation*}
after dividing both sides by $g(s_{k_2})>0$, which is situation \eqref{eq:evendeigencond1}.
\item If $K_2$ is empty, we can argue in an analogous way: the injectivity of $p$ on $[0,s^*]$ implies the existence of some $k_1\in K_1\ne\emptyset$ with $s_k=s_{k_1}$ for all $k\in K_1$ and thus
\begin{equation*}
s_k=\frac{1}{|K_1|},\quad k\in K_1,
\end{equation*}
which is situation \eqref{eq:evendeigencond2}.
\item Finally, assume that both $K_1$ and $K_2$ are nonempty. As in the previous special cases, the injectivity of $p$ on $[0,s^*]$ and on $(s^*,\infty)$ implies the existence of certain $k_1\in K_1$ and $k_2\in K_2$ with $s_k=s_{k_1}$ for all $k\in K_1$ and $s_k=s_{k_2}$ for all $k\in K_2$, such that
\begin{equation*}
\frac{g(s_{k_1})}{s_{k_1}}
=
\frac{g(s_{k_2})}{s_{k_2}}
=
|K_1|g(s_{k_1})+|K_2|g(s_{k_2}).
\end{equation*}
We have $s_{k_1}\ne s^*$ because of
\begin{equation*}
\frac{g(s_{k_1})}{s_{k_1}}=\frac{g(s_{k_2})}{s_{k_2}}>\frac{g(s^*)}{s^*}.
\end{equation*}
By using that $|K_1|s_{k_1}+|K_2|s_{k_2}=\sum_{j=1}^n s_j=1$, we observe that $s_{k_1}\in (0,s^*)$ is a zero of the even-degree polynomial
\begin{equation*}
r(s):=p(s)-p\Big(\frac{1-|K_1|s}{|K_2|}\Big)=\frac{g(s)}{s}-\frac{|K_2|g\big(\frac{1-|K_1|s}{|K_2|}\big)}{1-|K_1|s}
\end{equation*}
from \eqref{eq:evendr}, and we have
\begin{equation*}
s_{k_2}=\frac{1-|K_1|s_{k_1}}{|K_2|},
\end{equation*}
showing \eqref{eq:evendeigencond3} and \eqref{eq:evendeigencond3b}.
\end{itemize}
Conversely, assume that arbitrary disjoint subsets $K_1,K_2\subset\{1,\ldots,n\}$ are given, with $K_1\ne\emptyset$ or $K_2\ne\emptyset$. If $K_1=\emptyset$ or if $K_2=\emptyset$, setting $s_k\in[0,1]$ as in \eqref{eq:evendeigencond1} or in \eqref{eq:evendeigencond2} and following the proof of part (i) with $K$ replaced by $K_1$ or by $K_2$, respectively, we see that $\mathbf h(\mathbf s)=\mathbf 0$. If $K_1$ and $K_2$ are nonempty, and if $s_{k_1}\in(0,s^*)$ is an arbitrary zero of the polynomial $r$ from \eqref{eq:evendr}, such that $s_{k_2}$ from \eqref{eq:evendeigencond3b} is contained in $(s^*,1]$, we see that $s_k$ from \eqref{eq:evendeigencond3} fulfills
\begin{equation*}
s_k\ge 0,\quad 
\sum_{k=1}^n s_k=|K_1|s_{k_1}+|K_2|s_{k_2}=1,\quad
\sum_{k=1}^{n-1}s_k=1-s_n\le 1,
\end{equation*}
i.e., $\mathbf s:=(s_1,\ldots,s_{n-1})\in\Delta_{n-1}$. Moreover, regardless of whether $k\in K_1$, $k\in K_2$ or $k\notin K_1\cup K_2$, we have
\begin{align*}
s_k\sum_{j=1}^n g(s_j)
=
s_k\big(|K_1|g(s_{k_1})+|K_2|g(s_{k_2})\big)
=
g(s_k),
\end{align*}
so that $\mathbf h(\mathbf s)=\mathbf 0$.
\end{enumerate}
\end{proof}

\begin{remark}\label{rem:dodd}
In case that $d\ge 3$ is odd, the zeros of $\mathbf h$ given by \eqref{eq:hzeros} are precisely the midpoints of the unit simplex $\Delta_{n-1}$ and of all its lower-dimensional facets, see also Figure \ref{fig:Simpodd}. In particular, if $n=2$, we obtain that $h$ vanishes exactly at
\begin{equation*}
s_1\in\{0,{\textstyle\frac{1}{2}},1\}.
\end{equation*}
If $n=3$, $\mathbf h$ vanishes exactly at
\begin{equation*}
(s_1,s_2)\in\big\{(0,0),(1,0),(0,1),
({\textstyle\frac{1}{2}},0),
(0,{\textstyle\frac{1}{2}}),
({\textstyle\frac{1}{2}},{\textstyle\frac{1}{2}}),
({\textstyle\frac{1}{3}},{\textstyle\frac{1}{3}})\big\}.
\end{equation*}
\end{remark}

\begin{figure}[h]
\centering
\subfigure[Zeros of $\mathbf h$ if $n=2$ and $d$ odd]{
 \includegraphics[width=.4\linewidth]{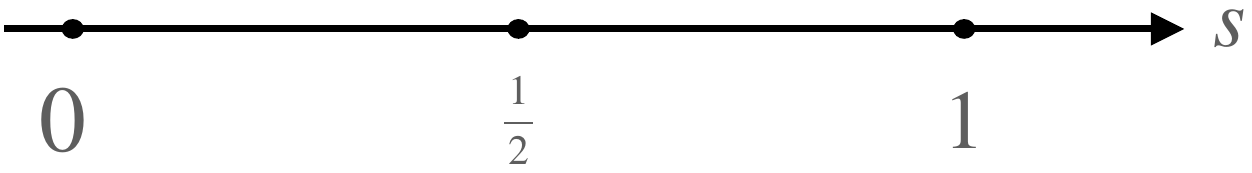}
}
\subfigure[Zeros of $\mathbf h$ if $n=3$ and $d$ odd]{
 \includegraphics[width=.4\linewidth]{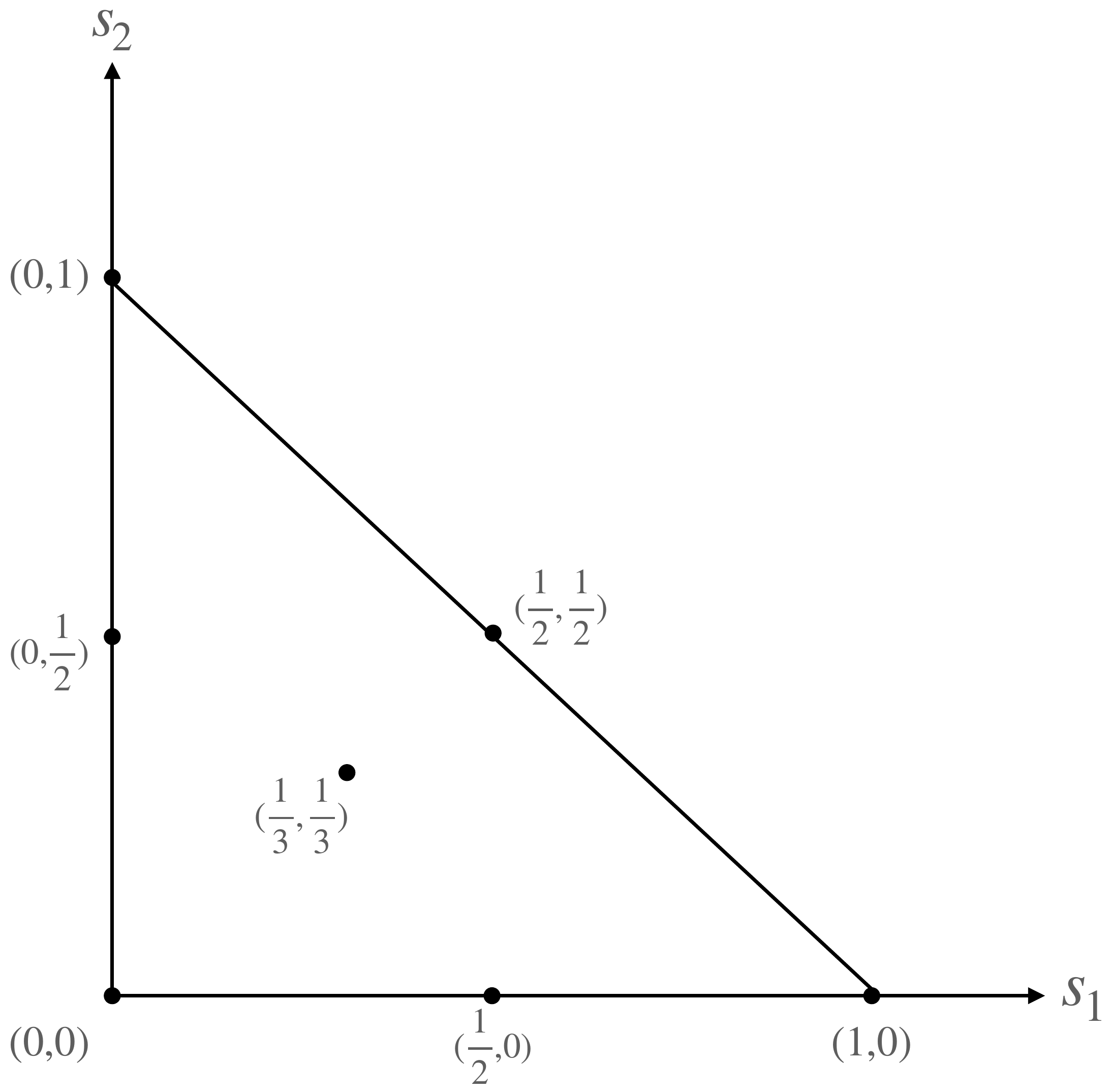}
}
\caption{Geometric interpretation of the zero set of $\mathbf{h}$ for $d$ odd}
\label{fig:Simpodd}
\end{figure}

\begin{remark}
If $d\ge 4$ is even, the situations \eqref{eq:evendeigencond1}, \eqref{eq:evendeigencond2} or \eqref{eq:evendeigencond3} can only occur if the respective conditions $\frac{1}{|K_2|}> s^*$, $\frac{1}{|K_1|}\le s^*$ or $s_{k_1}<s^*<s_{k_2}$ are fulfilled.
\end{remark}

\begin{ex}\label{ex:deven}
If $n=3$ and $d=4$, we compute that
\begin{equation*}
p(s)=\frac{g(s)}{s}
=\frac{(4s-1)^3+1}{s}
=64s^2-48s+12
\end{equation*}
is strictly convex with a unique global minimum at $s^*=\frac{3}{8}<\frac{1}{2}$. Therefore, situation \eqref{eq:evendeigencond1} occurs if and only if $K_1=\emptyset$ and $|K_2|\in\{1,2\}$ and thus
\begin{equation*}
K_2\in\big\{
\{1\},\{2\},\{3\},
\{1,2\},\{1,3\},\{2,3\}
\big\},
\end{equation*}
which corresponds to the subset of zeros
\begin{equation*}
\big\{(1,0),(0,1),(0,0),
({\textstyle\frac{1}{2},\frac{1}{2}}),
({\textstyle\frac{1}{2},0}),
({\textstyle 0,\frac{1}{2}})
\big\}
\end{equation*}
of
\begin{align*}
\mathbf h(\mathbf s)
&=
\begin{pmatrix}
s_1\big(g(s_1)+g(s_2)+g(1-s_1-s_2)\big)-g(s_1)\\
s_2\big(g(s_1)+g(s_2)+g(1-s_1-s_2)\big)-g(s_2)
\end{pmatrix}\\
&=
\begin{pmatrix}
-192s_1^3s_2-192s_1^2s_2^2+32s_1^3+288s_1^2s_2+96s_1s_2^2-48s_1^2-96s_1s_2+16s_1\\
-192s_1s_2^3-192s_1^2s_2^2+32s_2^3+288s_1s_2^2+96s_1^2s_2-48s_2^2-96s_1s_2+16s_2
\end{pmatrix}.
\end{align*}
The constraint $\frac{1}{|K_1|}\le s^*=\frac{3}{8}$ in situation \eqref{eq:evendeigencond2} can only be fulfilled if $K_1=\{1,2,3\}$ and $K_2=\emptyset$, which corresponds to the zero $({\textstyle\frac{1}{3},\frac{1}{3}})$ of $\mathbf h$. Finally, as concerns situation \eqref{eq:evendeigencond3}, the only three possible configurations are $(|K_1|,|K_2|)\in\{(1,1),(1,2),(2,1)\}$. If $|K_1|=|K_2|=1$, the polynomial $r$ from \eqref{eq:evendr} reads as
\begin{equation*}
r(s)=\frac{g(s)}{s}-\frac{g(1-s)}{1-s}=32s-16,
\end{equation*}
having the unique zero $s=\frac{1}{2}>s^*$, so that this case cannot occur. If $|K_1|=1$ and $|K_2|=2$, we obtain
\begin{equation*}
r(s)=\frac{g(s)}{s}-\frac{2g(\frac{1-s}{2})}{1-s}=48s^2-40s+8,
\end{equation*}
having the zeros $\frac{1}{3}<s^*$ and $\frac{1}{2}>s^*$. But since $s:=\frac{1}{3}$ would correspond to $\frac{1-s}{2}=\frac{1}{3}$ which is not strictly larger than $s^*$, this case cannot occur either. If $|K_1|=2$ and $|K_2|=1$, we obtain
\begin{equation*}
r(s)=\frac{g(s)}{s}-\frac{g(1-2s)}{1-2s}=-192s^2+112s-16,
\end{equation*}
having the zeros $\frac{1}{3}$ and $\frac{1}{4}$, both less than $s^*$. The first zero $s:=\frac{1}{3}$ of $r$ would correspond to $1-2s=\frac{1}{3}$, which is not strictly larger than $s^*$, which again is not allowed. The second zero $s_{k_1}:=\frac{1}{4}$ of $r$ yields the corresponding argument $s_{k_2}:=1-2s_{k_1}=\frac{1}{2}$ and hence induces the remaining zeros
\begin{equation*}
\big\{
({\textstyle \frac{1}{4},\frac{1}{4}}),
({\textstyle \frac{1}{4},\frac{1}{2}}),
({\textstyle \frac{1}{2},\frac{1}{4}})
\big\}
\end{equation*}
of $\mathbf h$, see also Figure \ref{fig:Simpeven}.
\end{ex}

\begin{figure}[h]
\centering
\subfigure[Zeros of $\mathbf h$ if $n=2$ and $d=4$ ]{
 \includegraphics[width=.4\linewidth]{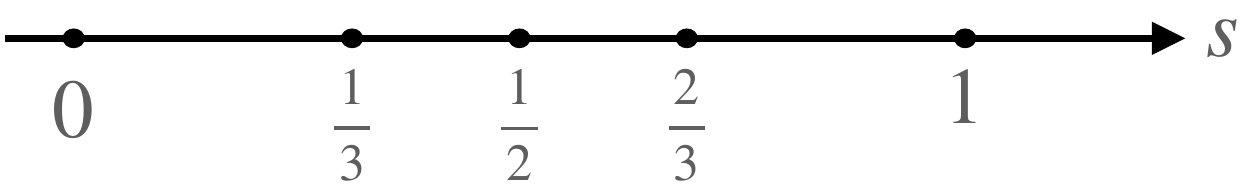}
}
\subfigure[Zeros of $\mathbf h$ if $n=3$ and $d=4$ ]{
 \includegraphics[width=.4\linewidth]{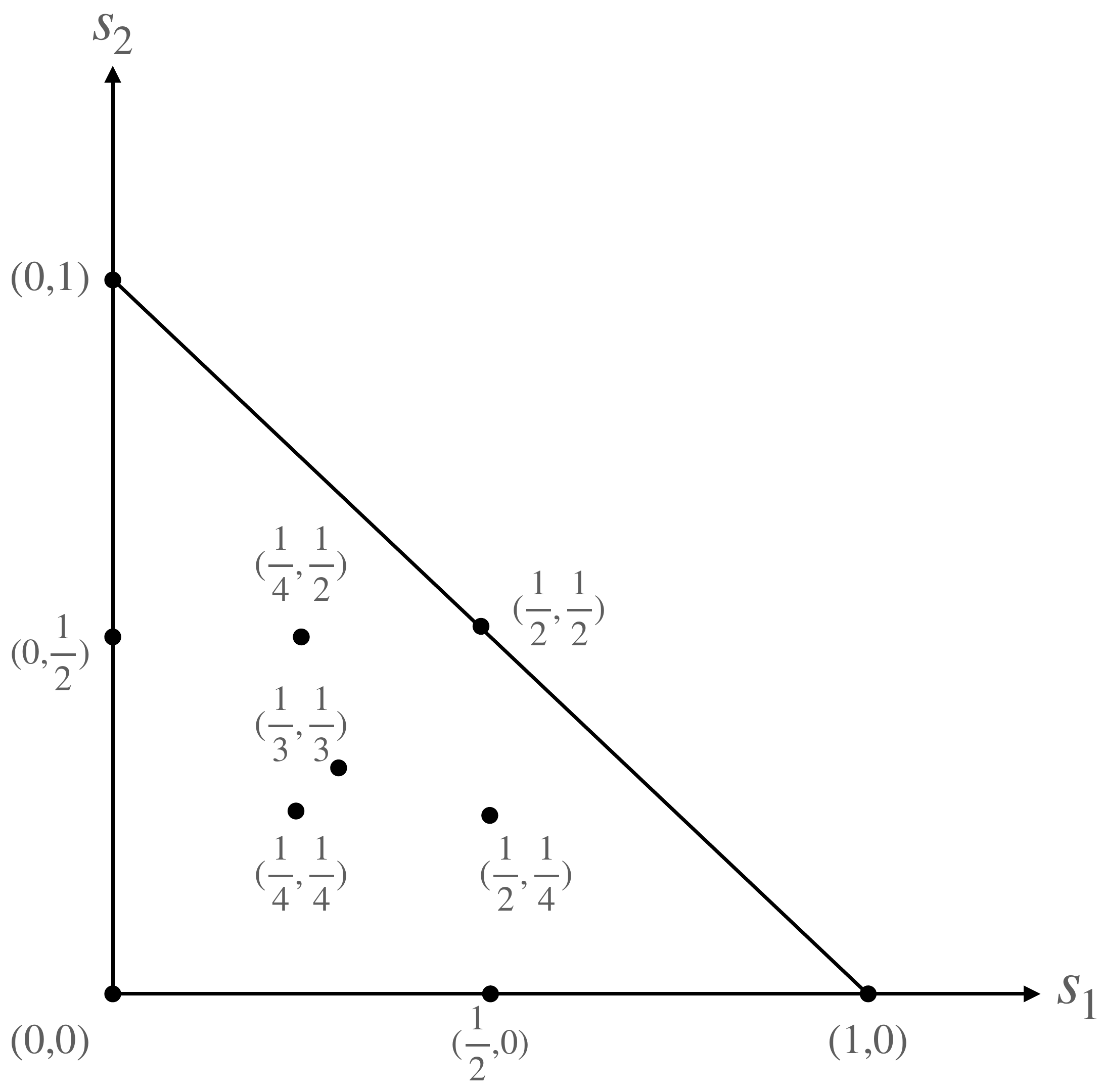}
}
\caption{Geometric interpretation of the zero set of $\mathbf{h}$ for $d=4$}
\label{fig:Simpeven}
\end{figure}

We are now in the position to enumerate all normalized eigenpairs of $\mathcal T$ in the generic case $n,d\ge 2$.
\begin{theo}\label{theo:eigenpairs}
Let $n,d\ge 2$, and let $\mathcal T=\sum_{k=1}^{n+1} \mathbf v_k^{\otimes d}$. 
\begin{enumerate}[(i)]
\item In the matrix case $d=2$, each $\mathbf v\in\mathds R^n$ with $\|\mathbf v\|=1$ is an eigenvector of $\mathcal T$, with positive eigenvalue
\begin{equation}\label{eq:eigenvaluematrixcase}
\mu=1+\frac{1}{n}.
\end{equation}
\item If $d\ge 3$ is odd, $\mathbf v\in\mathds R^n$ with $\|\mathbf v\|=1$ is an eigenvector of $\mathcal T$ if and only if there exists a nonempty subset $K\subset\{1,\ldots,n+1\}$ of cardinality at most $n$, such that
\begin{equation}\label{eq:eigenvectoroddd}
\mathbf v=\frac{\sum\limits_{k\in K}\mathbf v_k}{\Big\|\sum\limits_{k\in K}\mathbf v_k\Big\|}
=
\frac{\sum\limits_{k\in K}\mathbf v_k}{\sqrt{\frac{|K|(n+1-|K|)}{n}}},
\end{equation}
and the corresponding eigenvalue is given by
\begin{equation}\label{eq:eigenvalueoddd}
\mu=\frac{\big(n+1-|K|\big)^{d-1}-|K|^{d-1}}{n^{d/2}(|K|(n+1-|K|))^{d/2-1}}.
\end{equation}
The eigenvalue $0$ corresponds to the case $2|K|=n+1$. Therefore, if $n$ is even, all eigenvalues are different from $0$.
\item If $d\ge 4$ is even, $\mathbf v\in\mathds R^n$ with $\|\mathbf v\|=1$ is an eigenvector of $\mathcal T$ if and only if one of the following two conditions is met: Either there exists a nonempty subset $K\subset\{1,\ldots,n+1\}$ of cardinality at most $n$, such that $\mathbf v$ has the form \eqref{eq:eigenvectoroddd}, with positive eigenvalue
\begin{equation}\label{eq:eigenvalueevend1}
\mu=\frac{\big(n+1-|K|\big)^{d-1}+|K|^{d-1}}{n^{d/2}(|K|(n+1-|K|))^{d/2-1}}
\end{equation}
or there exist nonempty, disjoint subsets $K_1\subset\{1,\ldots,n+1\}$ and $K_2\subset\{1,\ldots,n+1\}$, each of cardinality at most $n$, such that with $s^*\in [\frac{1}{n},\frac{2}{n+1})$ from Lemma \ref{lemma:g}(v) and $0<s_{k_1}\le s^*<s_{k_2}\le 1$ with $|K_1|s_{k_1}+|K_2|s_{k_2}=1$ and $\frac{g(s_{k_1})}{s_{k_1}}=\frac{g(s_{k_2})}{s_{k_2}}$, we have
\begin{equation}\label{eq:eigenvectorevend2}
\mathbf v=\frac{s_{k_1}\sum\limits_{k\in K_1}\mathbf v_k+s_{k_2}\sum\limits_{k\in K_2}\mathbf v_k}{\Big\|s_{k_1}\sum\limits_{k\in K_1}\mathbf v_k+s_{k_2}\sum\limits_{k\in K_2}\mathbf v_k\Big\|}
=
\frac{s_{k_1}\sum\limits_{k\in K_1}\mathbf v_k+s_{k_2}\sum\limits_{k\in K_2}\mathbf v_k}{\sqrt{
\frac{(n+1)s_{k_1}^2|K_1|+(n+1)s_{k_2}^2|K_2|-1}{n}}}
\end{equation}
with positive eigenvalue
\begin{equation}\label{eq:eigenvalueevend2}
\mu=\frac{|K_1|((n+1)s_{k_1}-1)^d+|K_2|((n+1)s_{k_2}-1)^d+n+1-|K_1|-|K_2|}{n^{d/2}((n+1)s_{k_1}^2|K_1|+(n+1)s_{k_2}^2|K_2|-1)^{d/2}}.
\end{equation}
\end{enumerate}
\end{theo}
\begin{proof}
\begin{enumerate}[(i)]
\item If $d=2$, \eqref{eq:simplexframeoperator} tells us that $\mathcal T=(1+\frac{1}{n})\mathbf I$.
\item If $d\ge 3$ is odd, the first identity in \eqref{eq:eigenvectoroddd} follows by an application of Lemma \ref{lemma:cyclic} and Proposition \ref{propo:h}(iii). The second identity in \eqref{eq:eigenvectoroddd} can be verified by using \eqref{eq:simplexframegramian}, which yields
\begin{equation*}
\Big\|\sum_{k\in K}\mathbf v_k\Big\|^2
=
\sum_{k\in K}\sum_{j\in K}\langle\mathbf v_k,\mathbf v_j\rangle
=
\sum_{k\in K}\Big(1-\frac{|K|-1}{n}\Big)
=
\frac{|K|(n+1-|K|)}{n}.
\end{equation*}
As to the corresponding eigenvalue $\mu$ of $\mathbf v=\frac{\mathbf z}{\|\mathbf z\|}$, where $\mathbf z:=\sum_{k\in K}\mathbf v_k$, we can use that by the oddity of $d$,
\begin{align*}
\langle\mathcal T\cdot\mathbf z^{d-1},\mathbf z\rangle
&=
\sum_{k\in K}\Big\langle\sum_{j\in K}\mathbf v_j,\mathbf v_k\Big\rangle^d
+\sum_{\substack{1\le k\le n+1\\ k\notin K}}\Big\langle\sum_{j\in K}\mathbf v_j,\mathbf v_k\Big\rangle^d\\
&=
\sum_{k\in K}\Big(1-\frac{|K|-1}{n}\Big)^d+\sum_{\substack{1\le k\le n+1\\ k\notin K}}\Big({-}\frac{|K|}{n}\Big)^d\\
&=
\frac{|K|(n+1-|K|)}{n^d}\Big(\big(n+1-|K|\big)^{d-1}-|K|^{d-1}\Big),
\end{align*}
which yields that the eigenvalue of $\mathbf z$ is
\begin{equation*}
\frac{\langle\mathcal T\cdot\mathbf z^{d-1},\mathbf z\rangle}{\|\mathbf z\|^2}
=
\frac{\big(n+1-|K|\big)^{d-1}-|K|^{d-1}}{n^{d-1}},
\end{equation*}
from which we can deduce \eqref{eq:eigenvalueoddd} by an application of \eqref{eq:normalizedeigenpair}.
\item If $d\ge 4$ is even, by invoking Lemma \ref{lemma:cyclic}, the first family of normalized eigenvectors given by \eqref{eq:eigenvectoroddd} corresponds to the situations \eqref{eq:evendeigencond1} and \eqref{eq:evendeigencond2} from Proposition \ref{propo:h}. By using that $d$ is even, similar to the reasoning in (ii), we can compute that with $\mathbf z:=\sum_{k\in K}\mathbf v_k$, we have
\begin{equation*}
\langle\mathcal T\cdot\mathbf z^{d-1},\mathbf z\rangle=\frac{|K|(n+1-|K|)}{n^d}\Big(\big(n+1-|K|\big)^{d-1}+|K|^{d-1}\Big),
\end{equation*}
which yields \eqref{eq:eigenvalueevend1} after normalization. The second family of normalized eigenvectors given by \eqref{eq:eigenvectorevend2} corresponds to situation \eqref{eq:evendeigencond3}, and the second identity in \eqref{eq:eigenvectorevend2} follows via a similar argument as in (ii), by setting
\begin{equation*}
\mathbf z:=s_{k_1}\sum_{k\in K_1}\mathbf v_k+s_{k_2}\sum_{k\in K_2}\mathbf v_k
\end{equation*}
and by using \eqref{eq:simplexframegramian}, $K_1\cap K_2=\emptyset$ and $s_{k_1}|K_1|+s_{k_2}|K_2|=1$, that
\begin{align*}
\|\mathbf z\|^2
&=
s_{k_1}^2\Big\|\sum_{k\in K_1}\mathbf v_k\Big\|^2
+2s_{k_1}s_{k_2}\Big\langle\sum_{k\in K_1}\mathbf v_k,\sum_{j\in K_2}\mathbf v_j\Big\rangle+s_{k_2}^2\Big\|\sum_{k\in K_2}\mathbf v_k\Big\|^2\\
&=
s_{k_1}^2\frac{|K_1|(n+1-|K_1|)}{n}
-2s_{k_1}s_{k_2}\frac{|K_1||K_2|}{n}
+s_{k_2}^2\frac{|K_2|(n+1-|K_2|)}{n}\\
&=
\frac{1}{n}\Big((n+1)s_{k_1}^2|K_1|+(n+1)s_{k_2}^2|K_2|-1\Big).
\end{align*}
As concerns the corresponding eigenvalue $\mu$ of $\mathbf v=\frac{\mathbf z}{\|\mathbf z\|}$, we combine that $s_{k_1}|K_1|+s_{k_2}|K_2|=1$, $K_1\cap K_2=\emptyset$ and that $d$ is even, which yields the positive number
\begin{align*}
\lefteqn{\langle\mathcal T\cdot\mathbf z^{d-1},\mathbf z\rangle
=
\sum_{k\in K_1}\langle\mathbf z,\mathbf v_k\rangle^d
+
\sum_{k\in K_2}\langle\mathbf z,\mathbf v_k\rangle^d
+
\sum_{\substack{1\le k\le n+1\\ k\notin K_1\cup K_2}}\langle\mathbf z,\mathbf v_k\rangle^d}\\
&=
\sum_{k\in K_1}\Big(s_{k_1}\Big(1-\frac{|K_1|-1}{n}\Big)-s_{k_2}\frac{|K_2|}{n}\Big)^d\\
&\phantom{=}
+\sum_{k\in K_2}\Big(s_{k_2}\Big(1-\frac{|K_2|-1}{n}\Big)-s_{k_1}\frac{|K_1|}{n}\Big)^d\\
&\phantom{=}
+\sum_{\substack{1\le k\le n+1\\ k\notin K_1\cup K_2}}\Big({-}\frac{s_{k_1}|K_1|+s_{k_2}|K_2|}{n}\Big)^d\\
&=\frac{|K_1|((n+1)s_{k_1}-1)^d+|K_2|((n+1)s_{k_2}-1)^d+n+1-|K_1|-|K_2|}{n^d},
\end{align*}
and thus \eqref{eq:eigenvalueevend2} after normalization.
\end{enumerate}
\end{proof}
\subsection{Special case $n=2$}
In case that the local dimension $n$ is equal to $2$, the results from the previous generic analysis concretise as follows. The single barycentric coordinate $s\in[0,1]$ of an eigenvector $\mathbf v=s\mathbf v_1+(1-s)\mathbf v_2$ has to solve the nonlinear equation \eqref{eq:criticaln2}. We will therefore enumerate all real zeros of the expression
\begin{equation}\label{eq:hn2}
h(s):=(1-s)g(s)-sg(1-s)
=g(s)-s\big(g(s)+g(1-s)\big),\quad s\in\mathds R,
\end{equation}
in the following proposition.
\begin{propo}\label{propo:hn2}
Let $d\ge 2$. Then $h$ from \eqref{eq:hn2} is antisymmetric with respect to $s=\frac{1}{2}$, i.e.,
\begin{equation}\label{eq:hn2antisymmetry}
h(s)=-h(1-s),\quad s\in\mathds R.
\end{equation}
Moreover, depending on the parity of $d$, $h$ has the following properties.
\begin{enumerate}[(i)]
\item If $d$ is even, $h$ vanishes at least at $s\in\{0,\frac{1}{3},\frac{1}{2},\frac{2}{3},1\}$. In particular, $h$ is the zero polynomial if $d\in\{2,4\}$, and for each even $d\ge 2$, we can factorize $h(s)$ into
\begin{equation}\label{eq:hn2factorization}
h(s)=-9s(s-1)(2s-1)(3s-1)(3s-2)\sum_{\substack{p,q\ge 0\\ p+q\le d/2-3}}(3s-1)^{2p}(3s-2)^{2q},
\end{equation}
so that if $d\ge 6$ is even, $h$ does not have further real zeros than $\{0,\frac{1}{3},\frac{1}{2},\frac{2}{3},1\}$.
\item If $d$ is odd, $h$ vanishes at $s\in\{0,\frac{1}{2},1\}$, and we have $h(s)<0$ for $0<s<\frac{1}{2}$ and $h(s)>0$ for $\frac{1}{2}<s<1$.
\end{enumerate}
\end{propo}
\begin{proof}
$h$ obviously fulfills \eqref{eq:hn2antisymmetry}, which implies that $h(\frac{1}{2})=0$. Moreover, for each $d\ge 2$ we have $\varphi(0)=0$ by Lemma \ref{lemma:g}(i), which yields $h(0)={-}h(1)=0$.
\begin{enumerate}[(i)]
\item Let $d\ge 2$ be even. Then Lemma \ref{lemma:g}(i) implies that $g(\frac{2}{3})=2g(\frac{1}{3})$, so that $h(\frac{1}{3})=-h(\frac{2}{3})=0$. Therefore, $h$ vanishes at least at $s\in\{0,\frac{1}{3},\frac{1}{2},\frac{2}{3},1\}$. If $d\in\{2,4\}$, $h$ therefore has to vanish identically by the fundamental theorem of algebra. We will then prove \eqref{eq:hn2factorization} by induction over even dimensions $d\ge 2$. To this end, let us define
\begin{equation*}
h_k(s):=(1-s)\big((3s-1)^{2k-1}+1\big)-s\big((2-3s)^{2k-1}+1\big),\quad k=1,2,\ldots
\end{equation*}
which is equal to $h(s)$ from \eqref{eq:hn2} for $d=2k$. We already know from the previous reasoning that $h_1$ and $h_2$ vanish identically. Assume now that for some $k\ge 2$, the factorization
\begin{equation*}
h_k(s)=-9s(s-1)(2s-1)(3s-1)(3s-2)\sum_{\substack{p,q\ge 0\\ p+q\le k-3}}(3s-1)^{2p}(3s-2)^{2q}
\end{equation*}
holds true. Then we can compute that
\begin{align*}
h_{k+1}-h_k(s)
&=
(1-s)\big((3s-1)^{2k+1}-(3s-1)^{2k-1}\big)-s\big((2-3s)^{2k+1}-(2-3s)^{2k-1}\big)\\
&=
(1-s)(3s-1)^{2k-1}\big((3s-1)^2-1\big)-s(2-3s)^{2k-1}\big((2-3s)^2-1\big)\\
&=
-3s(s-1)(3s-1)(3s-2)\big((3s-1)^{2(k-1)}-(3s-2)^{2(k-1)}\big)\\
&=
-9s(s-1)(3s-1)(3s-2)(2s-1)\sum_{p=0}^{k-2}(3s-1)^{2p}(3s-2)^{2(k-2-p)},
\end{align*}
which yields the desired factorization of $h_{k+1}$, thereby proving \eqref{eq:hn2factorization} for all even $d\ge 2$.
If $d\ge 6$ is even, the trailing sum in \eqref{eq:hn2factorization} is nonempty and positive, so that $h$ does not have other real zeros than $\{0,\frac{1}{3},\frac{1}{2},\frac{2}{3},1\}$.
\item If $d=2k+1$ is odd, $k\ge 1$, Lemma \ref{lemma:g}(iv) tells us that $g$ is strictly convex. Therefore, 
\begin{equation*}
g(s)+g(1-s)\ge 2g\big({\textstyle\frac{1}{2}s+\frac{1}{2}(1-s)}\big)
=2g\big({\textstyle\frac{1}{2}}\big)=2^{1-2k}-2,\quad s\in\mathds R,
\end{equation*}
with equality if and only if $s=\frac{1}{2}$. On the other hand, by Lemma \ref{lemma:g}(v), we know that $s\mapsto\frac{g(s)}{s}$ is strictly increasing on $(0,\infty)$, so that
\begin{equation*}
\frac{g(s)}{s}
\le\frac{g(\frac{1}{2})}{\frac{1}{2}}
=2g({\textstyle\frac{1}{2}})=2^{1-2k}-2,\quad 0<s\le\frac{1}{2}.
\end{equation*} 
By combining both estimates, we obtain that
\begin{equation*}
g(s)+g(1-s)\ge\frac{g(s)}{s},\quad 0<s\le\frac{1}{2},
\end{equation*}
with equality only if $s=\frac{1}{2}$, thereby proving the claim because of the antisymmetry of $h$ with respect to $\frac{1}{2}$.
\end{enumerate}
\end{proof}

We are now in the position to enumerate all normalized eigenpairs of $\mathcal T$ in the case $n=2$.
\begin{theo}\label{theo:eigenpairsn2}
Let $d\ge 2$ and $n=2$, and let $\mathcal T=\sum_{k=1}^3 \mathbf v_k^{\otimes d}$. 
\begin{enumerate}[(i)]
\item If $d\in\{2,4\}$, each $\mathbf v\in\mathds R^2\setminus\{0\}$ with $\|\mathbf v\|_2=1$ is an eigenvector of $\mathcal T$, with positive eigenvalue
\begin{equation}\label{eq:eveneigenpairsn21}
\mu=\begin{cases}
\frac{3}{2},&d=2,\\
\frac{9}{8},&d=4.
\end{cases}
\end{equation}
\item If $d\ge 6$ is even, there are exactly $12$ normalized eigenpairs $(\mathbf v,\mu)$ of $\mathcal T$, given by 
\begin{equation}\label{eq:eveneigenpairsn22}
\big\{(\pm \mathbf v_k,1+2^{1-d}):1\le k\le 3\big\}\cup\big\{(\mathbf v_k+2\mathbf v_j)/\sqrt{3},3^{d/2}2^{1-d}):1\le k\ne j\le 3\big\},
\end{equation}
and the corresponding eigenvalues are positive.
\item If $d\ge 3$ is odd, there are exactly $6$ normalized eigenpairs $(\mathbf v,\mu)$ of $\mathcal T$, given by
\begin{equation}\label{eq:oddeigenpairsn2}
\big\{(\pm \mathbf v_k,\pm(1-2^{1-d})):1\le k\le 3\big\}.
\end{equation}
\end{enumerate}
\end{theo}
\begin{proof}
Proposition \ref{propo:hn2} yields all eigenvectors from the convex hull $\{s\mathbf v_1+(1-s)\mathbf v_2:0\le s\le 1\}$ of $\mathbf v_1$ and $\mathbf v_2$, which can then be mapped to normalized eigenvectors within the conical hull $\{s\mathbf v_1+t\mathbf v_2:s,t\ge 0\}$ by means of \eqref{eq:normalizedeigenpair}. All eigenvectors in the other three sectors of $\mathds R^2$ are then found by rotation, i.e., by cyclic shifts of $\mathbf v_1,\mathbf v_2,\mathbf v_3$.

\begin{enumerate}[(i)]
\item If $d\in\{2,4\}$, Proposition \ref{propo:hn2}(i) shows that each $\mathbf v\in\mathds R^2\setminus\{0\}$ is an eigenvector. If $d=2$, we can represent each $\mathbf v\in\mathds R^2\setminus\{0\}$ with $\|\mathbf v\|_2=1$ as $\mathbf v=\alpha \mathbf v_1+\beta \mathbf v_2$ with $\alpha^2-\alpha\beta+\beta^2=1$. By using that $\langle \mathbf v_j,\mathbf v_k\rangle$ is equal to $1$ if $j=k$ and equal to ${-}\frac{1}{2}$ if $j\ne k$, and by using $\mathbf v_1+\mathbf v_2+\mathbf v_3=\mathbf 0$, we can compute
\begin{align*}
\mathcal T\cdot\mathbf v
=
\sum_{k=1}^3\langle \mathbf v,\mathbf v_k\rangle \mathbf v_k
&=
\Big(\alpha-\frac{\beta}{2}\Big)\mathbf v_1
+\Big(\beta-\frac{\alpha}{2}\Big)\mathbf v_2
+\Big({-}\frac{\alpha}{2}-\frac{\beta}{2}\Big)\mathbf v_3\\
&=
\frac{3}{2}\alpha \mathbf v_1+\frac{3}{2}\beta \mathbf v_2\\
&=
\frac{3}{2}\mathbf v.
\end{align*}
If $d=4$, a similar computation yields
\begin{align*}
\mathcal T\cdot\mathbf v^3
=
\sum_{k=1}^3\langle \mathbf v,\mathbf v_k\rangle^3 \mathbf v_k
&=
\Big(\alpha-\frac{\beta}{2}\Big)^3 \mathbf v_1
+\Big(\beta-\frac{\alpha}{2}\Big)^3 \mathbf v_2
+\Big({-}\frac{\alpha}{2}-\frac{\beta}{2}\Big)^3 \mathbf v_3\\
&=
\frac{9}{8}(\alpha^3-\alpha^2\beta+\alpha\beta^2)\mathbf v_1
+
\frac{9}{8}(\beta^3-\beta^2\alpha+\beta\alpha^2)\mathbf v_2\\
&=
\frac{9}{8}\mathbf v.
\end{align*}
\item If $d\ge 6$ is even, Proposition \ref{propo:hn2}(i) tells us that all normalized eigenvectors from the conical hull of $\mathbf v_1$ and $\mathbf v_2$ are given by
\begin{equation*}
\mathbf v_1,\quad
\frac{\mathbf v_1+2\mathbf v_2}{\sqrt{3}},\quad
\mathbf v_1+\mathbf v_2=-\mathbf v_3,\quad
\frac{2\mathbf v_1+\mathbf v_2}{\sqrt{3}},\quad
\mathbf v_2.
\end{equation*}
The corresponding eigenvalues can be computed by using \eqref{eq:mu3} and \eqref{eq:normalizedeigenpair}, they read as
\begin{equation*}
1+2^{1-d},\quad
3^{d/2}2^{1-d},\quad
1+2^{1-d},\quad
3^{d/2}2^{1-d},\quad
1+2^{1-d},
\end{equation*}
respectively.
\item If $d\ge 3$ is odd, we proceed in a similar way as in (ii). By Proposition \ref{propo:hn2}(ii), the normalized eigenvectors from the conical hull of $\mathbf v_1$ and $v_2$ are given by
\begin{equation*}
\mathbf v_1,\quad
\mathbf v_1+\mathbf v_2=-\mathbf v_3,\quad
\mathbf v_2,
\end{equation*}
and the corresponding eigenvalues
\begin{equation*}
1-2^{1-d},\quad
{-}1+2^{1-d},\quad
1-2^{1-d},
\end{equation*}
respectively, can be inferred from  \eqref{eq:mu3} and \eqref{eq:normalizedeigenpair}.
\end{enumerate}
\end{proof}

\begin{coro}
If $n=2$ and  $d\ge 3$ is odd, $\mathcal T=\sum_{k=1}^3 \mathbf v_k^{\otimes d}$ is not odeco.
\end{coro}
\begin{proof}
If $d\ge 3$ is odd, there is no orthogonal set $\{\mathbf w_1,\mathbf w_2\}\subset\mathds R^2$ of eigenvectors, so that $\mathcal T$ cannot be orthogonally decomposable.
\end{proof}

\subsection{Special case $n=3$}
If $n=3$, the eigenstructure of the simplex tensor $\mathcal T=\sum_{k=1}^4\mathbf v_k^{\otimes d}$ is slightly more complicated as in the previous case $n=2$, but still amenable to a concrete analysis. From Proposition \ref{propo:h} and Theorem \ref{theo:eigenpairs}, we can deduce the following theorem.
\begin{theo}\label{theo:eigenpairsn3}
Let $d\ge 2$ and $n=3$, and let $\mathcal T=\sum_{k=1}^4 \mathbf v_k^{\otimes d}$. 
\begin{enumerate}[(i)]
\item If $d=2$, each $\mathbf v\in\mathds R^3$ with $\|\mathbf v\|=1$ is an eigenvector of $\mathcal T$, with eigenvalue
\begin{equation}\label{eq:eigenvaluematrixcasen3}
\mu=\frac{4}{3}.
\end{equation}
\item If $d\ge 3$ is odd, $\mathbf v\in\mathds R^3$ with $\|\mathbf v\|=1$ is an eigenvector of $\mathcal T$ if and only if there exists a nonempty subset $K\subset\{1,\ldots,4\}$ of cardinality at most $3$, such that
\begin{equation}\label{eq:eigenvectorodddn3}
\mathbf v=\frac{\sum\limits_{k\in K}\mathbf v_k}{\Big\|\sum\limits_{k\in K}\mathbf v_k\Big\|}
=
\frac{\sum\limits_{k\in K}\mathbf v_k}{\sqrt{\frac{|K|(4-|K|)}{3}}},
\end{equation}
and the corresponding eigenvalue is given by
\begin{equation}\label{eq:eigenvalueodddn3}
\mu=\frac{\big(4-|K|\big)^{d-1}-|K|^{d-1}}{3^{d/2}(|K|(4-|K|))^{d/2-1}}.
\end{equation}
We have $\mu=0$ if and only if $|K|=2$, with three linearly independent eigenvectors
\begin{equation}\label{eq:zeroeigenvectorsodddn3}
\frac{\mathbf v_1+\mathbf v_2}{\sqrt{\frac{4}{3}}},\quad 
\frac{\mathbf v_1+\mathbf v_3}{\sqrt{\frac{4}{3}}},\quad 
\frac{\mathbf v_1+\mathbf v_4}{\sqrt{\frac{4}{3}}}.
\end{equation}
\item If $d\ge 4$ is even, $\mathbf v\in\mathds R^3$ with $\|\mathbf v\|=1$ is an eigenvector of $\mathcal T$ if and only if one of the following two conditions is met: Either there exists a nonempty subset $K\subset\{1,\ldots,4\}$ of cardinality at most $3$, such that $\mathbf v$ has the form \eqref{eq:eigenvectorodddn3}, with positive eigenvalue
\begin{equation}\label{eq:eigenvalueevend1n3}
\mu=\frac{\big(4-|K|\big)^{d-1}+|K|^{d-1}}{3^{d/2}(|K|(4-|K|))^{d/2-1}}.
\end{equation}
or there exist nonempty, disjoint subsets $K_1,K_2\subset\{1,\ldots,4\}$, such that $K_1\cup K_2$ has cardinality at most $3$, and with $s^*\in [\frac{1}{3},\frac{1}{2})$ from Lemma \ref{lemma:g}(v) and $0<s_{k_1}\le s^*<s_{k_2}\le 1$ with $|K_1|s_{k_1}+|K_2|s_{k_2}=1$ and $\frac{g(s_{k_1})}{s_{k_1}}=\frac{g(s_{k_2})}{s_{k_2}}$, we have
\begin{equation}\label{eq:eigenvectorevend2n3}
\mathbf v=\frac{s_{k_1}\sum\limits_{k\in K_1}\mathbf v_k+s_{k_2}\sum\limits_{k\in K_2}\mathbf v_k}{\Big\|s_{k_1}\sum\limits_{k\in K_1}\mathbf v_k+s_{k_2}\sum\limits_{k\in K_2}\mathbf v_k\Big\|}
=
\frac{s_{k_1}\sum\limits_{k\in K_1}\mathbf v_k+s_{k_2}\sum\limits_{k\in K_2}\mathbf v_k}{\sqrt{
\frac{4s_{k_1}^2|K_1|+4s_{k_2}^2|K_2|-1}{3}}}
\end{equation}
with positive eigenvalue
\begin{equation}\label{eq:eigenvalueevend2n3}
\mu=\frac{|K_1|(4s_{k_1}-1)^d+|K_2|(4s_{k_2}-1)^d+4-|K_1|-|K_2|}{3^{d/2}(4s_{k_1}^2|K_1|+4s_{k_2}^2|K_2|-1)^{d/2}}.
\end{equation}
\end{enumerate}
\end{theo}
\begin{proof}
\begin{enumerate}[(i)]
\item This is part (i) of Theorem \ref{theo:eigenpairs}, and we have $\mathcal T=\frac{4}{3}\mathbf I$.
\item \eqref{eq:eigenvectorodddn3} and \eqref{eq:eigenvalueodddn3} directly follow from part (ii) of Theorem \ref{theo:eigenpairs}. We have $\mu=0$ if and only if $4-|K|=|K|$, i.e., $|K|=2$. This yields $\binom{4}{2}=6$ possibilities for $K$,
\begin{equation*}
K\in\big\{\{1,2\},\{1,3\},\{1,4\},\{2,3\},\{2,4\},\{3,4\}\big\},
\end{equation*}
but due to $\sum_{k=1}^4\mathbf v_k=\mathbf 0$, there are three collinear pairs
\begin{equation*}
(\mathbf v_1+\mathbf v_2,\mathbf v_3+\mathbf v_4),\quad
(\mathbf v_1+\mathbf v_3,\mathbf v_2+\mathbf v_4),\quad
(\mathbf v_1+\mathbf v_4,\mathbf v_2+\mathbf v_3)
\end{equation*}
of eigenvectors, from which we can deduce the linearly independent set \eqref{eq:zeroeigenvectorsodddn3} of normalized eigenvectors with zero eigenvalue.
\item This is part (iii) of Theorem \ref{theo:eigenpairs}.
\end{enumerate}
\end{proof}

\begin{ex}
If $n=3$ and $d=4$, Theorem \ref{theo:eigenpairsn3} parametrizes two families of normalized eigenpairs. On the one hand, the $2^4-2=14$ nonempty subsets
\begin{equation*}
\begin{split}
K\in\big\{&\{1\},\{2\},\{3\},\{4\},\{1,2\},\{1,3\},\{1,4\},\{2,3\},\{2,4\},\{3,4\},\\
&\{1,2,3\},\{1,2,4\},\{1,3,4\},\{2,3,4\}\big\}
\end{split}
\end{equation*}
of $\{1,\ldots,4\}$ with at most $3$ elements induce the normalized eigenvectors
\begin{equation*}
\begin{split}
\mathbf v\in\Bigg\{
&\mathbf v_1,\mathbf v_2,\mathbf v_3,\mathbf v_4,
\frac{\mathbf v_1+\mathbf v_2}{\sqrt{\frac{4}{3}}},
\frac{\mathbf v_1+\mathbf v_3}{\sqrt{\frac{4}{3}}},
\frac{\mathbf v_1+\mathbf v_4}{\sqrt{\frac{4}{3}}},
\frac{\mathbf v_2+\mathbf v_3}{\sqrt{\frac{4}{3}}},
\frac{\mathbf v_2+\mathbf v_4}{\sqrt{\frac{4}{3}}},
\frac{\mathbf v_3+\mathbf v_4}{\sqrt{\frac{4}{3}}},\\
&
\mathbf v_1+\mathbf v_2+\mathbf v_3,
\mathbf v_1+\mathbf v_2+\mathbf v_4,
\mathbf v_1+\mathbf v_3+\mathbf v_4,
\mathbf v_2+\mathbf v_3+\mathbf v_4
\Bigg\}
\end{split}
\end{equation*}
from \eqref{eq:eigenvectorodddn3}, which by means of $\sum_{k=1}^4\mathbf v_k=\mathbf 0$ could be grouped into 7 pairs of collinear eigenvectors. The respective eigenvalues read as
\begin{equation*}
\mu\in\Big\{
\frac{28}{27},
\frac{28}{27},
\frac{28}{27},
\frac{28}{27},
\frac{4}{9},
\frac{4}{9},
\frac{4}{9},
\frac{4}{9},
\frac{4}{9},
\frac{4}{9},
\frac{28}{27},
\frac{28}{27},
\frac{28}{27},
\frac{28}{27}
\Big\}.
\end{equation*}
On the other hand, a second family of normalized eigenpairs is given by pairs of nonempty, disjoint subsets $K_1,K_2\subset\{1,\ldots,4\}$ with $|K_1|+|K_2|\le 3$ and the constraints listed in Theorem \ref{theo:eigenpairsn3}(iii). There are $3$ possibilities for the cardinalities $|K_j|$ of $K_j$, namely
\begin{equation*}
\big(|K_1|,|K_2|\big)
\in
\big\{(1,1),(1,2),(2,1)\big\}.
\end{equation*}
By retracing Example \ref{ex:deven}, we see that only the case $|K_1|=2$ and $|K_2|=1$ and the corresponding zeros
\begin{equation*}
({\textstyle\frac{1}{4},\frac{1}{4}}),\quad
({\textstyle\frac{1}{4},\frac{1}{2}}),\quad
({\textstyle\frac{1}{2},\frac{1}{4}})
\end{equation*}
of $\mathbf h$ from \eqref{eq:h} induce further normalized eigenvectors, namely
\begin{equation*}
\begin{split}
\mathbf v\in\Bigg\{
&\frac{2\mathbf v_1+\mathbf v_2+\mathbf v_3}{\sqrt{\frac{8}{3}}},
\frac{2\mathbf v_1+\mathbf v_2+\mathbf v_4}{\sqrt{\frac{8}{3}}},
\frac{2\mathbf v_1+\mathbf v_3+\mathbf v_4}{\sqrt{\frac{8}{3}}},
\frac{2\mathbf v_2+\mathbf v_1+\mathbf v_3}{\sqrt{\frac{8}{3}}},\\
&\frac{2\mathbf v_2+\mathbf v_1+\mathbf v_4}{\sqrt{\frac{8}{3}}},
\frac{2\mathbf v_2+\mathbf v_3+\mathbf v_4}{\sqrt{\frac{8}{3}}},
\frac{2\mathbf v_3+\mathbf v_1+\mathbf v_2}{\sqrt{\frac{8}{3}}},
\frac{2\mathbf v_3+\mathbf v_1+\mathbf v_4}{\sqrt{\frac{8}{3}}},\\
&\frac{2\mathbf v_3+\mathbf v_2+\mathbf v_4}{\sqrt{\frac{8}{3}}},
\frac{2\mathbf v_4+\mathbf v_1+\mathbf v_2}{\sqrt{\frac{8}{3}}},
\frac{2\mathbf v_4+\mathbf v_1+\mathbf v_3}{\sqrt{\frac{8}{3}}},
\frac{2\mathbf v_4+\mathbf v_2+\mathbf v_3}{\sqrt{\frac{8}{3}}}
\Bigg\},
\end{split}
\end{equation*}
which can be grouped into 6 pairs of collinear eigenvectors, with eigenvalue
\begin{equation*}
\mu=\frac{8}{9}
\end{equation*}
each.
\end{ex}

\newpage

\section{Robustness analysis}\label{robustanalysis}
In this section we study the robustness of all normalized eigenvectors of a regular simplex tensor $\mathcal T=\sum_{k=1}^{n+1}\mathbf v_k^{\otimes d}$ with respect to the forward mapping of the tensor power teration 
\begin{equation}\label{eq:phi}
\varphi(\mathbf x):=\frac{\mathcal T \cdot \mathbf x^{d-1}}{\| \mathcal T \cdot \mathbf x^{d-1} \|},\quad \mathbf x\in\mathds R^n,\quad \mathcal T \cdot \mathbf x^{d-1} \ne 0,
\end{equation}
at least if $n=2$ and $n=3$.

In order to assess the attractivity of a given fixed point $\mathbf x\in\mathds R^n$ of the differentiable mapping $\varphi$, we will determine the spectral radius of the Jacobian of $\varphi$ at $\mathbf x$.
We will check whether the Jacobian of $\varphi$ at $\mathbf x$ has spectral radius strictly less than one, which indicates local contractivity of $\varphi$, or strictly greater than one, which indicates local expansivity at least in one direction.

It is well-known, see also \cite{MRU21}, that the Jacobian of $\varphi$ at an arbitrary vector $\mathbf x\in\mathds R^n$ is given by the symmetric matrix
\begin{equation}\label{eq:phideriv}
\varphi'(\mathbf x)
=
\frac{d-1}{\big\| \mathcal T \cdot \mathbf x^{d-1} \big\|}
\left(
I-\frac{(\mathcal T \cdot \mathbf x^{d-1}) (\mathcal T \cdot \mathbf x^{d-1})^\top}{\big\| \mathcal T \cdot \mathbf x^{d-1} \big\|^2}
\right) \mathcal T \cdot \mathbf x^{d-2},
\end{equation}
where
\begin{equation}\label{eq:Txdminus2}
\mathcal T\cdot\mathbf x^{d-2}:=
\Big(\sum_{i_1,\ldots,i_{d-2}=1}^n T_{i_1,\ldots,i_{d-2},j,k}x_{i_1}\cdots x_{i_{d-2}}x_jx_k\Big)_{1\le j,k\le n}
\end{equation}
is the $(d-2)$-fold partial contraction of $\mathcal T$ with $\mathbf x^{d-2}$. 
In the matrix case $d=2$ and $\mathcal T=\mathbf A$, \eqref{eq:phideriv} is to be understood as
\begin{equation}\label{eq:phideriv2}
\varphi'(\mathbf x)
=
\frac{1}{\|\mathbf A\mathbf x\|}
\left(
I-\frac{\mathbf A\mathbf x(\mathbf A\mathbf x)^\top}{\|\mathbf A\mathbf x\|^2}
\right)\mathbf A.
\end{equation}
If $\mathbf x\in\mathds R^n$ is a normalized eigenvector of $\mathcal T$ with eigenvalue $\mu$, we obtain that
\begin{equation}\label{eq:phideriveigenvector}
\varphi'(\mathbf x)
=
\frac{d-1}{|\mu|}(I-\mathbf x\mathbf x^\top) \mathcal T \cdot \mathbf x^{d-2} 
=
\frac{d-1}{|\mu|}\big( \mathcal T \cdot \mathbf x^{d-2} -\mu \mathbf x\mathbf x^\top\big),
\end{equation}
which for $d=2$ and $\mathcal T=\mathbf A$ reduces to
\begin{equation}\label{eq:phideriveigenvector2}
\varphi'(\mathbf x)
=
\frac{1}{|\mu|}(\mathbf I-\mathbf x\mathbf x^\top)\mathbf A.
\end{equation}

Before we start with the eigenvalue analysis of $\varphi'$ at a normalized tensor eigenvector $\mathbf x\in\mathds R^n$, we note that by the very structure of \eqref{eq:phideriveigenvector}, we have
\begin{equation}\label{eq:phideriveigenvectorkernel}
\varphi'(\mathbf x)\mathbf x= \mathbf{0}.
\end{equation}
Therefore, the spectral radius $\rho(\varphi'(\mathbf x))$ of the symmetric matrix $\varphi'(\mathbf x)$ is determined by the other eigenvectors of $\varphi'(\mathbf x)$ which are contained in the orthogonal complement $\spn\{\mathbf x\}^\perp$.

\subsection{Case $n=2$}
In case that the local dimension $n$ is equal to two, we will now explicitly compute the spectral radius of the Jacobian of $\varphi$ at each normalized eigenvector. 
\begin{theo}
Let $d\ge 2$, and let $\mathcal T=\sum_{k=1}^3 \mathbf v_k^{\otimes d}$.
\begin{enumerate}[(i)]
\item If $d\in\{2,4\}$, the spectral radius of $\varphi'$ at all normalized $\mathbf x\in\mathds R^2$ is equal to 1, i.e., there are no robust eigenvectors.
\item If $d\ge 6$ is even, the spectral radius of $\varphi'$ at the normalized eigenvectors $\pm \mathbf v_k$ is equal to $\frac{3(d-1)}{2^{d-1}+1}$, i.e., these normalized eigenvectors are robust. The spectral radius of $\varphi'$ at the normalized eigenvectors $(\mathbf v_k+2\mathbf v_j)/\sqrt{3}$, $k\ne j$ is equal to $\frac{d-1}{3}$, i.e., those normalized eigenvectors are non-robust. 
\item If $d\ge 3$ is odd, the spectral radius of $\varphi'$ at all normalized eigenvectors $\pm \mathbf v_k$ is equal to $\frac{3(d-1)}{2^{d-1}-1}$, i.e., all normalized eigenvectors are non-robust for $d=3$, and robust for $d\ge 5$.
\end{enumerate}
\end{theo}
\begin{proof}
\begin{enumerate}[(i)]
\item If $d=2$, we have
\begin{equation*}
\mathcal T=\sum_{k=1}^3 \mathbf v_k^{\otimes 2}=\frac{3}{2}\mathbf I.
\end{equation*}
Therefore, each $\mathbf x\in\mathds R^2\setminus\{\mathbf 0\}$ with $\|\mathbf x\|=1$ is an eigenvector of $\mathcal T$ with eigenvalue $\mu=\frac{3}{2}$. An application of \eqref{eq:phideriv2} yields
\begin{equation*}
\varphi'(\mathbf x)=\frac{1}{3/2}(\mathbf I-\mathbf x\mathbf x^\top)\frac{3}{2}\mathbf I=\mathbf I-\mathbf x\mathbf x^\top,
\end{equation*}
which has eigenvalues $0$ and $1$, and thus spectral radius 1.

If $d=4$, Theorem \ref{theo:eigenpairsn2}(i) tells us that each $\mathbf x=(x_1,x_2)\in\mathds R^2\setminus\{0\}$ with $x_1^2+x_2^2=1$ is an eigenvector of $\mathcal T$, and hence of $\mathcal T\cdot\mathbf x^{d-2}\in\mathds R^{2\times 2}$, with eigenvalue $\mu=\frac{9}{8}$. By the structure of $\varphi'(\mathbf x)$, it follows that $\varphi'(\mathbf x)\mathbf x=\mathbf 0$. Setting $\mathbf w:=(x_2,-x_1)$, we compute that
\begin{align*}
(\mathcal T\cdot\mathbf x^2)\mathbf w
&=
\sum_{k=1}^3 \langle \mathbf x,\mathbf v_k\rangle^2\langle \mathbf w,\mathbf v_k\rangle \mathbf v_k\\
&=
x_1^2x_2
\begin{pmatrix}
1\\
0
\end{pmatrix}
+\Big({-}\frac{x_1}{2}+\frac{x_2\sqrt{3}}{2}\Big)^2\Big({-}\frac{x_2}{2}-\frac{x_1\sqrt{3}}{2}\Big)
\begin{pmatrix}
-\frac{1}{2}\\
\frac{\sqrt{3}}{2}
\end{pmatrix}\\
&\phantom{=}
+
\Big({-}\frac{x_1}{2}-\frac{x_2\sqrt{3}}{2}\Big)^2\Big({-}\frac{x_2}{2}+\frac{x_1\sqrt{3}}{2}\Big)
\begin{pmatrix}
-\frac{1}{2}\\
-\frac{\sqrt{3}}{2}
\end{pmatrix}\\
&=
\begin{pmatrix}
\frac{3}{8}x_1^2x_2+\frac{3}{8}x_2^3\\
{-}\frac{3}{8}x_1^3-\frac{3}{8}x_1x_2^3
\end{pmatrix}\\
&=
\frac{3}{8}\mathbf w,
\end{align*}
i.e., $\mathbf w$ is an eigenvector of $\mathcal T\cdot\mathbf x^2$ with eigenvalue $\frac{3}{8}$. By the orthogonality between $\mathbf x$ and $\mathbf w$, we obtain $\varphi'(\mathbf x)\mathbf w=\frac{8}{3}(\mathbf I-\mathbf x\mathbf x^\top)(\mathcal T\cdot\mathbf x^2)\mathbf w=\mathbf w$, so that $\varphi'(\mathbf x)$ has eigenvalues $0$ and $1$, and thus spectral radius 1.
\item If $d\ge 6$ is even, Theorem \ref{theo:eigenpairsn2}(i) tells us that there are two types of normalized eigenvectors, namely $\pm \mathbf v_k$ and $(\mathbf v_k+2\mathbf v_j)/\sqrt{3}$ for $1\le k\ne j\le 3$, with eigenvalues $\mu_1=1+2^{1-d}$ and $\mu_2=3^{d/2}2^{1-d}$, respectively. By using that $\sum_{k=1}^3 \mathbf v_k\mathbf v_k^\top=\frac{3}{2}\mathbf I$, we get
\begin{align*}
\mathcal T\cdot(\pm \mathbf v_k)^{d-2}
&=
\sum_{j=1}^3\langle \mathbf v_k,\mathbf v_j\rangle^{d-2}\mathbf v_j\mathbf v_j^\top\\
&=
\mathbf v_k\mathbf v_k^\top+\frac{1}{2^{d-2}}\sum_{j\ne k}\mathbf v_j\mathbf v_j^\top\\
&=
\Big(1-\frac{1}{2^{d-2}}\Big)\mathbf v_k\mathbf v_k^\top+\frac{3}{2^{d-1}}\mathbf I.
\end{align*}
Therefore, $\mathcal T\cdot(\pm \mathbf v_k)^{d-2}$ has the normalized eigenvectors $\mathbf v_k$ and $\mathbf w\in\spn\{\mathbf v_k\}^\perp$ with eigenvalues $\mu_1$ and $\frac{3}{2^{d-1}}$, respectively. Hence, $\mathcal T\cdot(\pm \mathbf v_k)^{d-2}$ and $\mathbf I-\mathbf v_k\mathbf v_k^\top$ having the same eigenvectors, $\varphi'(\pm \mathbf v_k)$ has the eigenvalues 
\begin{equation*}
0,\quad \frac{d-1}{|\mu_1|}\frac{3}{2^{d-1}}=\frac{3(d-1)}{2^{d-1}+1}.
\end{equation*}
The spectral radius of $\varphi'(\pm \mathbf v_k)$ is therefore given by
\begin{equation*}
\rho\big(\varphi'(\pm \mathbf v_k)\big)=\frac{3(d-1)}{2^{d-1}+1},
\end{equation*}
which is strictly less than $1$ if and only if $d\ge 6$.

As concerns the second eigenvector family, we proceed in an analogous way. If $1\le k\ne j\le 3$ are given, and if $r\in\{1,2,3\}$ is the index with $\{1,2,3\}\setminus\{k,j\}=\{r\}$, we have
\begin{equation*}
\langle \mathbf v_k+2\mathbf v_j,\mathbf v_l\rangle=\begin{cases}
0,&l=k,\\
\frac{3}{2},&l=j,\\
{-}\frac{3}{2},&l=r,
\end{cases}
\end{equation*}
which implies that
\begin{align*}
\mathcal T\cdot((\mathbf v_k+2\mathbf v_j)/\sqrt{3})^{d-2}\big\rangle
&=
3^{1-d/2}\sum_{l=1}^3\langle \mathbf v_k+2\mathbf v_j,\mathbf v_l\rangle^{d-2}\mathbf v_l\mathbf v_l^\top\\
&=
\frac{3^{d/2-1}}{2^{d-2}}(\mathbf v_j\mathbf v_j^\top+\mathbf v_r\mathbf v_r^\top)\\
&=
\frac{3^{d/2}}{2^{d-1}}\mathbf I-\frac{3^{d/2-1}}{2^{d-2}}\mathbf v_k\mathbf v_k^\top.
\end{align*}
Therefore, $\mathcal T\cdot((\mathbf v_k+2\mathbf v_j)/\sqrt{3})^{d-2}$ has the normalized eigenvectors $(\mathbf v_k+2\mathbf v_j)/\sqrt{3}\perp \mathbf v_k$ and $\mathbf v_k$ with eigenvalues $\mu_2$ and $\frac{3^{d/2-1}}{2^{d-1}}$, respectively. Hence, $\mathcal T\cdot((\mathbf v_k+2\mathbf v_j)/\sqrt{3})^{d-2}$ and $\mathbf I-(\mathbf v_k+2\mathbf v_j)(\mathbf v_k+2\mathbf v_j)^\top/3$ having the same eigenvectors, $\varphi'((\mathbf v_k+2\mathbf v_j)/\sqrt{3})$ has the eigenvalues
\begin{equation*}
0,\quad \frac{d-1}{|\mu_2|}\frac{3^{d/2-1}}{2^{d-1}}=\frac{d-1}{3}.
\end{equation*}
The spectral radius of $\varphi'((\mathbf v_k+2\mathbf v_j)/\sqrt{3})$ is therefore equal to $\frac{d-1}{3}>1$, so that the second family of normalized eigenvectors is not robust.
\item If $d\ge 3$ is odd, Theorem \ref{theo:eigenpairsn2}(ii) tells us that the only normalized eigenvectors of $\mathcal T$ are given by $\pm \mathbf v_k$, with eigenvalues $\mu_{\pm}=\pm(1-2^{1-d})$. By using that $\sum_{k=1}^3 \mathbf v_k\mathbf v_k^\top=\frac{3}{2}\mathbf I$, similar to the reasoning in (ii), we get
\begin{align*}
\mathcal T\cdot(\pm \mathbf v_k)^{d-2}
&=
\pm\sum_{j=1}^3\langle \mathbf v_k,\mathbf v_j\rangle^{d-2}\mathbf v_j\mathbf v_j^\top\\
&=
\pm\Big(\mathbf v_k\mathbf v_k^\top-\frac{1}{2^{d-2}}\sum_{j\ne k}\mathbf v_j\mathbf v_j^\top\Big)\\
&=
\pm\left(\Big(1+\frac{1}{2^{d-2}}\Big)\mathbf v_k\mathbf v_k^\top-\frac{3}{2^{d-1}}\mathbf I\right).
\end{align*}
Therefore, $\mathcal T\cdot (\pm \mathbf v_k)^{d-2}\rangle$ has the normalized eigenvectors $\mathbf v_k$ and $\mathbf w\in\spn\{\mathbf v_k\}^\perp$ with eigenvalues $\mu_{\pm}$ and $\mp\frac{3}{2^{d-1}}$, respectively. Hence, $\mathcal T\cdot(\pm \mathbf v_k)^{d-2}\rangle$ and $\mathbf I-\mathbf v_k\mathbf v_k^\top$ having the same eigenvectors, $\varphi'(\pm \mathbf v_k)$ has the eigenvalues
\begin{equation*}
0,\quad \mp\frac{d-1}{|\mu_\pm|}\frac{3}{2^{d-1}}=\mp\frac{3(d-1)}{2^{d-1}-1}.
\end{equation*}
The spectral radius of $\varphi'(\pm \mathbf v_k)$ is therefore given by
\begin{equation*}
\rho\big(\varphi'(\pm \mathbf v_k)\big)=\frac{3(d-1)}{2^{d-1}-1},
\end{equation*}
which is strictly less than $1$ if and only if $d\ge 5$.
\end{enumerate}
\end{proof}

\begin{ex}
In order to visualize the typical performance of the tensor power method in the two-dimensional case, let us proceed as follows. Each point on the unit sphere in $\mathbf R^2$ is uniquely determined by its angle (i.e., its argument), and we assign a color to each possible angle. For each starting point $\mathbf v^{(0)}$ on the unit sphere, we iterate the tensor power method \eqref{eq:tpm} until convergence, and then paint $\mathbf v^{(0)}$ with the color of its limit point. The resulting domains of attraction in the special case $d=7$ are visualized in Figure \ref{fig:attractionn2}, together with the spanning vectors $\mathbf{v}_{1}$, $\mathbf{v}_{2}$ and $\mathbf v_{3}$ of the regular simplex tensor $\mathcal T=\sum_{k=1}^3\mathbf v_k^{\otimes d}$.
\begin{figure}[h]
\centering
\includegraphics[width=.4\linewidth]{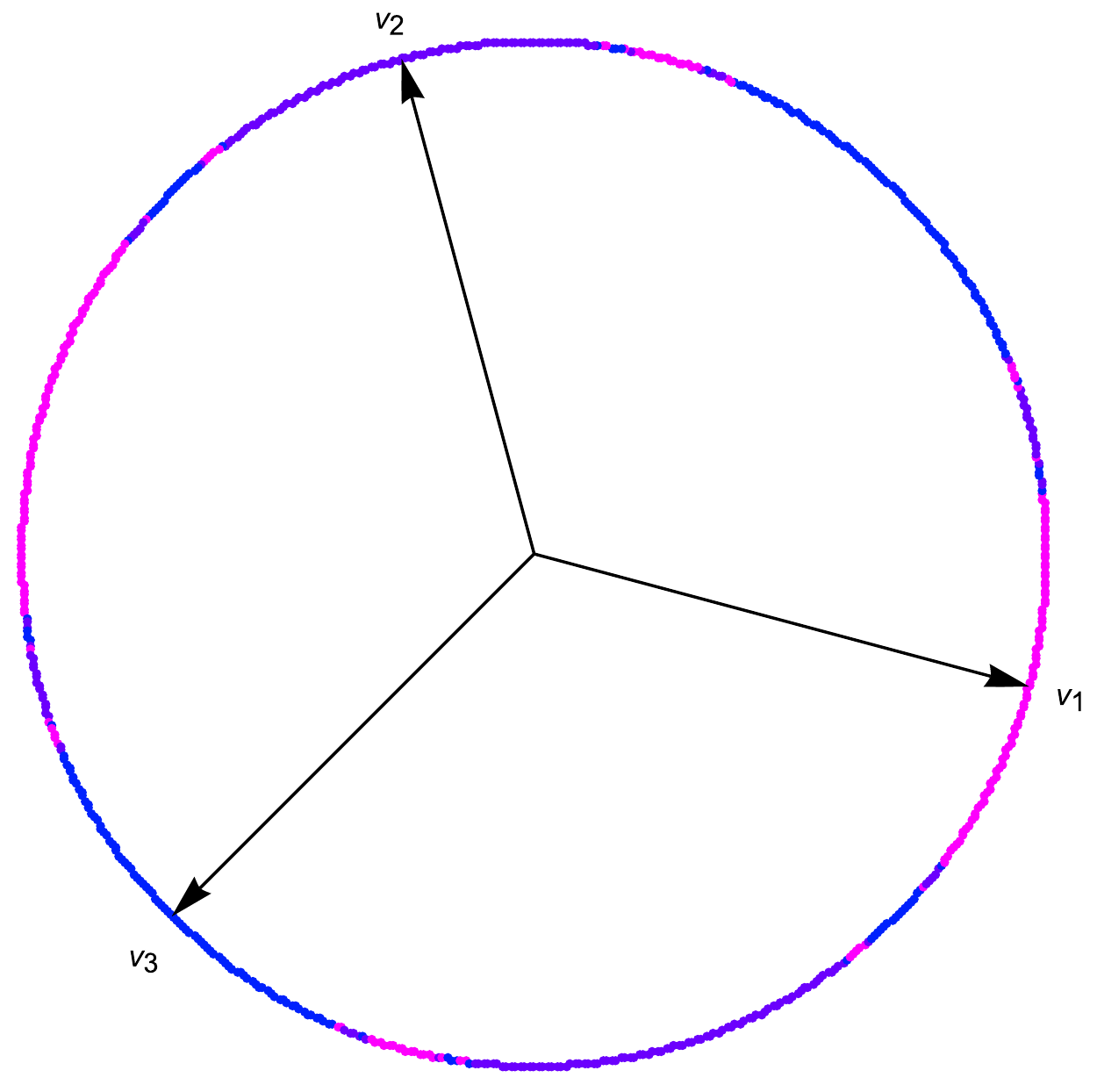}
\caption{Regions of attraction of the power tensor method if $n=2$ and $d=7$.}
\label{fig:attractionn2}
\end{figure}
It can be clearly seen that there are six regions of attraction which moreover contain the 
vectors $\pm\mathbf{v}_1, \pm\mathbf{v}_2, \pm\mathbf{v}_3$ as their respective centres. This observation is backed by Theorem \ref{theo:eigenpairsn2}(iii), which states that in this case $d=7$, all normalized eigenvectors of $\mathcal T$ are robust.
\end{ex}

\subsection{Case $n=3$}
In the special case $n=3$, we conduct two numerical experiments to assess the robustness of a normalized eigenvector.

The first experiment is set up in a similar way as in the two-dimensional case. We assign a unique color to each point on the unit sphere in $\mathds R^3$, depending on the respective spherical angles. In Figure \ref{fig:attractionn3}, we visualize the results of the tensor power iteration in the cases $d=6$ and $d=7$.
\begin{figure}[h]
\centering
\subfigure[$d=6$]{
  \includegraphics[width=.4\linewidth]{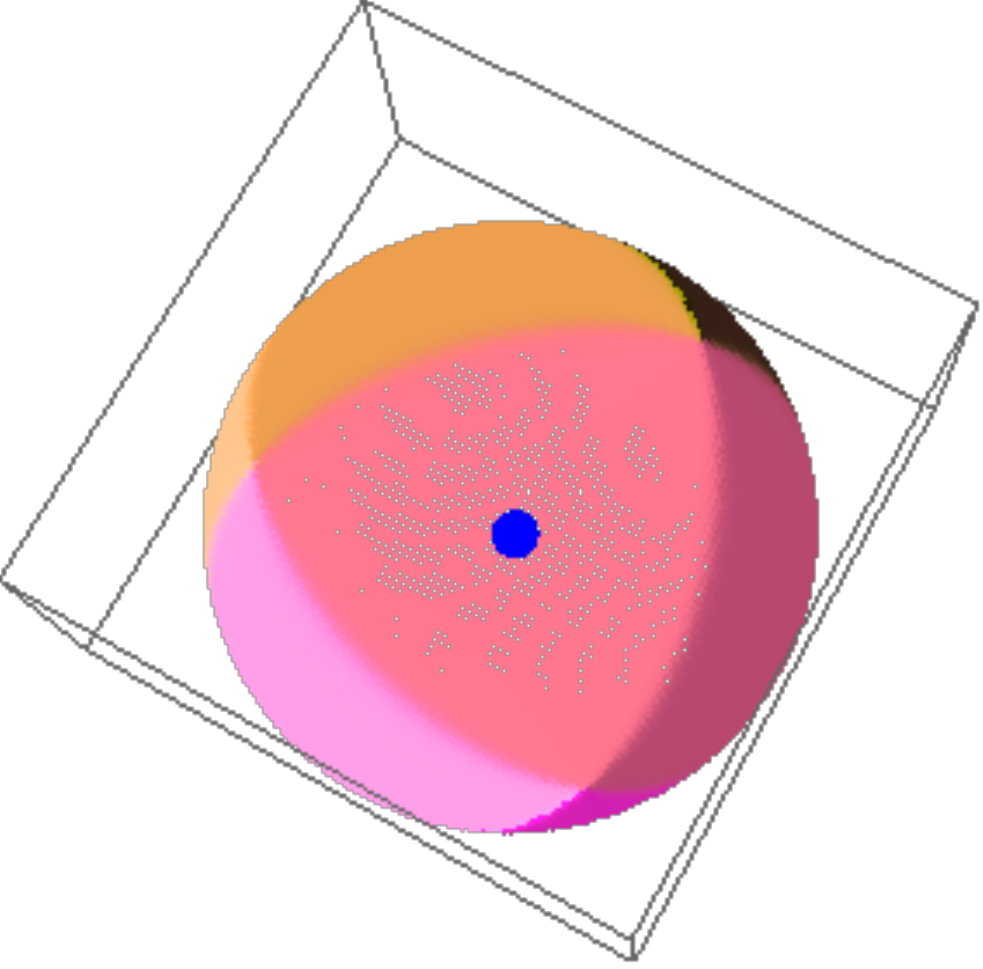}
}
\subfigure[$d=7$]{
  \includegraphics[width=.4\linewidth]{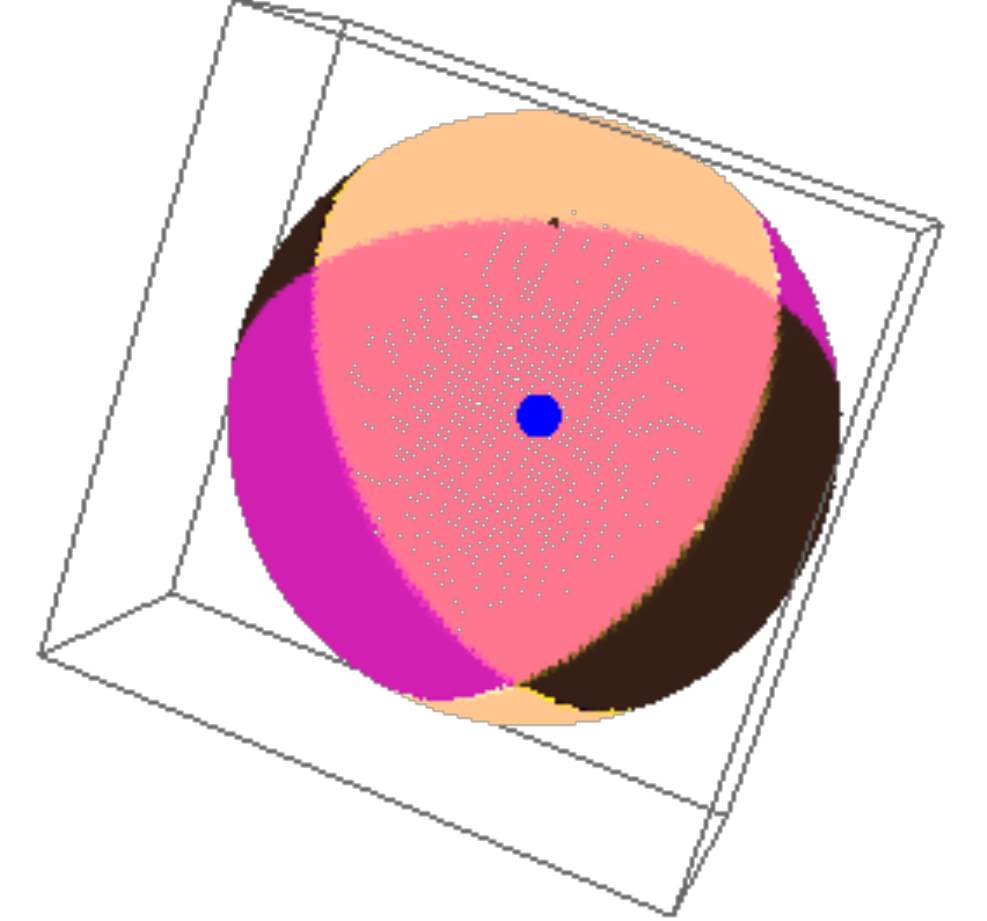}
}
\caption{Attraction points for $n=3$ and $d\in\{6,7\}$}
\label{fig:attractionn3}
\end{figure}
In each scenario, we can observe several regions of convergence which correspond to the robust eigenvectors. The blue dot in the raspberry-coloured region is the position of the eigenvector $\mathbf v_1$. Each eigenvector $\mathbf{v}_k$ corresponds to two regions, as ${-}\mathbf v_k$ is also an eigenvector.

As a second numerical experiment, we computed the spectral radii of the Jacobians $\varphi'(\mathbf x)$ at all normalized eigenvectors $\mathbf x$ of the regular simplex tensor $\mathcal T$ for a variety of mode parameters $d$. The results can be found in Table \ref{tab:experiments}. As we can see, there are non-robust eigenvectors even for small values of $d$. Moreover, we can observe that if $d$ is odd, there are always three eigenvectors with eigenvalue zero (cf. Theorem \ref{theo:eigenpairsn3}(ii)).

\begin{table}[h]
    \centering
    \begin{tabular}{|c|c|c|c|}
    \hline
         $d$ & $\mathbf{s} = (s_{1},s_{2})$ & $\rho( \varphi ' (\mathbf{s}))$ & $\text{comment}$ \\ \hline
         \rule{0pt}{0.5\normalbaselineskip}
         \multirow{7}{*}{$d=3$} 
         & $(0,0)$ & $2$ & not robust\\
         & $(1,0)$ & $2$ & not robust\\
         & $(0,1)$ & $2$ & not robust\\
         & $(\frac{1}{2},0)$ & $--$ & not defined as $\mu = 0$\\[0.5ex]
         & $(0, \frac{1}{2})$ & $--$ & not defined as $\mu = 0$\\[0.5ex]
         & $(\frac{1}{2},\frac{1}{2})$ & $--$ & not defined as $\mu = 0$\\[0.5ex]
         & $(\frac{1}{3},\frac{1}{3})$ & $2$ &  not robust\\[0.5ex]
         \hline
         \rule{0pt}{0.8\normalbaselineskip}
         \multirow{10}{*}{$d=4$} 
         & $(0,0)$ & $\frac{5}{7}$ & robust\\[0.5ex]
         & $(1,0)$ & $\frac{5}{7}$ & robust\\[0.5ex] 
         & $(0,1)$ & $\frac{5}{7}$ & robust\\[0.5ex] 
         & $(\frac{1}{2},0)$ & $5$ & not robust\\
         & $(0, \frac{1}{2})$ & $5$ & not robust\\
         & $(\frac{1}{2},\frac{1}{2})$ & $5$ & not robust\\[0.5ex]
         & $(\frac{1}{3},\frac{1}{3})$ & $\frac{5}{7}$ & robust\\[0.5ex]
         & $(\frac{1}{2},\frac{1}{4})$ & $\frac{5}{2}$ & not robust\\[0.5ex]
         & $(\frac{1}{4},\frac{1}{2})$ & $\frac{5}{2}$ & not robust\\[0.5ex]
         & $(\frac{1}{4},\frac{1}{4})$ & $\frac{5}{2}$ & not robust\\[0.5ex]
         \hline
         \rule{0pt}{0.8\normalbaselineskip}
         \multirow{7}{*}{$d=5$} 
         & $(0,0)$ & $\frac{3}{10}$ & robust\\[0.5ex]
         & $(1,0)$ & $\frac{3}{10}$ & robust\\[0.5ex] 
         & $(0,1)$ & $\frac{3}{10}$ & robust\\[0.5ex] 
         & $(\frac{1}{2},0)$ & $--$ & not defined as $\mu = 0$\\[0.5ex]
         & $(0, \frac{1}{2})$ & $--$ & not defined as $\mu = 0$\\[0.5ex]
         & $(\frac{1}{2},\frac{1}{2})$ & $--$ & not defined as $\mu = 0$\\[0.5ex]
         & $(\frac{1}{3},\frac{1}{3})$ & $\frac{3}{10}$ & robust\\[0.5ex]
         \hline 
         \rule{0pt}{0.8\normalbaselineskip}
         \multirow{10}{*}{$d=6$} 
         & $(0,0)$ & $\frac{7}{61}$ & robust\\[0.5ex]
         & $(1,0)$ & $\frac{7}{61}$ & robust\\[0.5ex] 
         & $(0,1)$ & $\frac{7}{61}$ & robust\\[0.5ex]
         & $(\frac{1}{2},0)$ & $7$ & not robust\\
         & $(0, \frac{1}{2})$ & $7$ & not robust\\
         & $(\frac{1}{2},\frac{1}{2})$ & $7$ & not robust\\
         & $(\frac{1}{3},\frac{1}{3})$ & $\frac{7}{61}$ & robust\\[0.5ex]
         & $(\frac{1}{2},\frac{1}{4})$ & $\frac{7}{2}$ & not robust\\[0.5ex]
         & $(\frac{1}{4},\frac{1}{2})$ & $\frac{7}{2}$ & not robust\\[0.5ex]
         & $(\frac{1}{4},\frac{1}{4})$ & $\frac{7}{2}$ & not robust\\[0.5ex]
         \hline
         \rule{0pt}{0.8\normalbaselineskip}
         \multirow{7}{*}{$d=7$} 
         & $(0,0)$ & $\frac{4}{91}$ & robust\\[0.5ex]
         & $(1,0)$ & $\frac{4}{91}$ & robust\\[0.5ex]
         & $(0,1)$ & $\frac{4}{91}$ & robust\\[0.5ex]
         & $(\frac{1}{2},0)$ & $--$ & not defined as $\mu = 0$\\[0.5ex]
         & $(0, \frac{1}{2})$ & $--$ & not defined as $\mu = 0$\\[0.5ex]
         & $(\frac{1}{2},\frac{1}{2})$ & $--$ & not defined as $\mu = 0$\\[0.5ex]
         & $(\frac{1}{3},\frac{1}{3})$ & $\frac{4}{91}$ & robust\\[0.5ex]
         \hline 
         \multicolumn{4}{@{}p{11cm}}{}
    \end{tabular}
    \caption{Numerical experiments for $n=3$ and mode numbers $d \in \{3,4,5,6,7\}$}
    \label{tab:experiments}
\end{table}

\section{Conclusion}
In this work we have been concerned with the eigenstructure analysis of real symmetric tensors. In the special case of regular simplex tensors, where all weights $\lambda_k$ in the symmetric decomposition \eqref{eq:symmdecompintro} are equal and the $\mathbf v_k$ are induced by $n+1$ equiangular vectors in $\mathds R^n$, we have seen that some normalized eigenpairs are attractive fixed points with respect to the tensor power iteration, whereas others are repelling. Therefore, as long as the tensor power iteration is used without modification, some eigenvectors may not be detectable numerically. These observations induce several directions of further research.

On the one hand, maintaining the viewpoint of the orthodox tensor power iteration, it would be a natural generalization to study symmetric tensors whose symmetric decomposition \eqref{eq:symmdecompintro} uses different weights $\lambda_k$ and/or is induced by a set of more than $n+1$ equiangular vectors $\mathbf v_k$ or even a generic tight frame. Here, one might aim at a complete characterization of those symmetric tensors which do have repelling eigenvectors, or whose normalized eigenvectors are given by the vectors $\mathbf v_k$ alone. Partial answers to the latter question have been given in \cite{MRU21}. Furthermore, the case of complex-valued tensors seems to be open and nontrivial, because even the existence and construction of equiangular tight frames are delicate tasks in large dimensions (cf. \cite{CRT08,SU07}).

On the other hand, algorithmic modifications could be employed to recover those eigenvectors that are non-robust under the orthodox tensor power iteration. Shifted tensor power iterations have been considered in \cite{KM11}, with good success, including a characterization of which normalized eigenpairs can and cannot be found numerically by such a scheme. But further modifications seem necessary to obtain a full numerical tensor eigenvalue solver, which, however, would exceed the scope of this paper.



\section*{Acknowledgments}
The authors would like to thank the University of Siegen for enabling our computations through the OMNI cluster. JS was supported by the Deutsche Forschungsgemeinschaft (DFG, German Research Foundation, project numbers 447948357 and 440958198), the Sino-German Center for Research Promotion (Project M-0294), the ERC (Consolidator Grant 683107/TempoQ) and the House of Young Talents of the University of Siegen.

\clearpage

\printbibliography
\end{document}